\documentclass[10pt]{amsart}
\usepackage[colorlinks=true,linkcolor=blue,citecolor=magenta]{hyperref}
\usepackage{graphics, bbm, bigints, relsize, fullpage}
\usepackage{amssymb,amsmath,amsthm, mathrsfs}
\usepackage{color}
\usepackage{tikz}
\usepackage[mathscr]{euscript}
\usepackage{enumerate}
\usepackage[conditional, light, first,bottomafter]{draftcopy}
\draftcopyName{DRAFT\space\today}{130}
\draftcopySetScale{65}

\usepackage{lmodern}
\usepackage{microtype}
\makeatletter
\renewcommand{\section}{\@startsection%
{section}%
{1}%
{0em}%
{1.7em}%
{1.2em}%
{\normalfont\large\centering\bfseries}}
\renewcommand{\@seccntformat}[1]%
{\csname the#1\endcsname.\hspace{0.5em}}
\makeatother
\renewcommand{\thesection}{\arabic{section}}

\numberwithin{equation}{section}
\renewcommand\appendix{\par
\setcounter{section}{0}%
\setcounter{subsection}{0}%
\setcounter{theorem}{0}
\setcounter{table}{0}
\setcounter{figure}{0}
\gdef\thetable{\Alph{table}}
\gdef\thefigure{\Alph{figure}}
\section*{Appendix}
\gdef\thesection{\Alph{section}}
\setcounter{section}{1}}
\newtheorem{theorem}{Theorem}[section]
\newtheorem{proposition}[theorem]{Proposition}
\newtheorem{lemma}[theorem]{Lemma}
\newtheorem{corollary}[theorem]{Corollary}
\theoremstyle{definition}
\newtheorem{definition}{Definition}
\newtheorem{remark}{Remark}

\newcommand{\reals}{\mathbb{R}}
\newcommand{\ie}{\emph{i.e.\kern.2em}}
\newcommand{\cf}{\emph{cf.\kern.2em}}
\newcommand{\viz}{\emph{viz.\kern.2em}}
\newcommand{\eg}{\emph{e.g.\kern.2em}}
\newcommand{\nats}{\mathbb{N}}

\newcommand{\complex}{\mathbb{C}}

\newcommand{\norm}[1]{\left\|#1\right\|}
\newcommand{\I}{{\rm i}}
\newcommand{\inner}[2]{\left\langle#1,#2\right\rangle}

\newcommand{\cH}{{\mathcal H}}
\newcommand{\cc}[1]{\overline{#1}}

\DeclareMathOperator{\re}{Re}
\DeclareMathOperator{\im}{Im}
\DeclareMathOperator{\dom}{dom}

\DeclareMathOperator{\ran}{ran}

\DeclareMathOperator*{\Span}{span}

\usepackage{stackengine}

\begin{document}

\title[Functional model for boundary-value problems]{Functional model for boundary-value problems}
\author{Kirill D. Cherednichenko}
\address{Department of Mathematical Sciences, University of Bath, Claverton Down, Bath, BA2 7AY, United Kingdom}
\email{cherednichenkokd@gmail.com}
\author{Alexander V. Kiselev}
\address{Departamento de F\'{i}sica Matem\'{a}tica, Instituto de Investigaciones en Matem\'aticas Aplicadas y en Sistemas, Universidad Nacional Aut\'onoma de M\'exico, C.P. 04510, M\'exico D.F. {\sc and} International Research Laboratory ``Multiscale Model Reduction'', Ammosov North-Eastern Federal University, Yakutsk, Russia}
\email{alexander.v.kiselev@gmail.com}
\author{Luis O. Silva}
\address{Departamento de F\'{i}sica Matem\'{a}tica, Instituto de Investigaciones en Matem\'aticas Aplicadas y en Sistemas, Universidad Nacional Aut\'onoma de M\'exico, C.P. 04510, M\'exico D.F. {\sc and} Department of Mathematical Sciences, University of Bath, Claverton Down, Bath, BA2 7AY, United Kingdom}
\email{silva@iimas.unam.mx}

\subjclass[2010]{47A45, 47F05, 35P25}

\keywords{Functional model;
Extensions of symmetric operators; Generalised boundary triples;
Boundary value problems; Spectrum}

\begin{abstract}
We develop a functional model for operators arising in the study of boundary-value
  problems of materials science and mathematical physics. We then provide explicit formulae for the resolvents of the associated extensions of symmetric operators in terms of
 appropriate Dirichlet-to-Neumann maps, which can be utilised in the analysis of the properties of parameter-dependent problems, including the study of their spectra.
\end{abstract}

\maketitle

\section{Introduction}

The need to understand and quantify the behaviour of solutions to
problems of mathematical physics has been central in driving the
development of theoretical tools for the analysis of boundary-value
problems (BVP). On the other hand, the second part of the last century
witnessed several substantial advances in the abstract methods of
spectral theory in Hilbert spaces, stemming from the groundbreaking
achievement of John von Neumann in laying the mathematical foundations
of quantum mechanics. Some of these advances have made their way into
the broader context of mathematical physics
\cite{MR0217440,Anderson,BetheSommerfeld}. In spite of these obvious
successes of spectral theory applied to concrete problems, the
operator-theoretic understanding of BVP has been lacking. However, in
models of short-range interactions, the idea of replacing the original
complex system by an explicitly solvable one, with a zero-radius potential (possibly with an internal structure), has proved to be highly
valuable
\cite{Berezin_Faddeev,Pavlov_internal_structure,Pavlov_Helmholtz_resonator, MR0080271,MR0024574,MR0024575, MR0051404}. This
facilitated an influx of methods of the theory of extensions (both
self-adjoint and non-selfadjoint) of symmetric operators to problems
of mathematical physics, culminating in the theory of boundary
triples.

The theory of boundary triples introduced in \cite{Gor, MR1087947, MR0365218, MR0592863}
has been successfully applied to the spectral analysis of BVP for
ordinary differential operators and related setups, {\it e.g.} that of
finite ``quantum graphs'', where the Dirichlet-to-Neumann maps act on
finite-dimensional ``boundary" spaces, see
\cite{CherednichenkoKiselevSilva} and references therein. However, in
its original form this theory is not suited for dealing with BVP for partial differential equations (PDE), see \cite[Section
7]{BMNW2008} for a relevant discussion. The key obstacle to such analysis is the
lack of boundary traces $\Gamma_0u$ and $\Gamma_1u$ for functions $u:\Omega\to{\mathbb R}$ (where $\Omega$ is a bounded open set with a smooth boundary) in the domain of the maximal operator $A$ corresponding to the differential expression considered ({\it e.g.} the operator $-\Delta$ on the domain of $L^2(\Omega)$-functions $u$ such that $\Delta u$ is in $L^2(\Omega)$) entering the Green identity
\begin{equation*}
\langle Au,v \rangle_{L^2(\Omega)} - \langle u,Av \rangle_{L^2(\Omega)} = \langle \Gamma_1 u, \Gamma_0 v \rangle_{L^2(\partial\Omega)}- \langle \Gamma_0 u,\Gamma_1 v  \rangle_{L^2(\partial\Omega)},\qquad u, v\in{\rm dom}(A),
\end{equation*}
in other words ${\rm dom}(A)\not\subset{\rm dom}(\Gamma_0)\cap{\rm dom}(\Gamma_1).$
Recently, when the works \cite{Grubb_Robin,Grubb_mixed,BehrndtLanger2007,Gesztesy_Mitrea, Ryzh_spec, BMNW2008} started to appear, it has transpired that, suitably modified, the boundary triples approach nevertheless admits a natural generalisation to the BVP setup, see also the seminal contributions by
M.\,S.\,Birman \cite{Birman}, L.\,Boutet de Monvel \cite{BdeM}, M.\,S.\,Birman and M.\,Z.\,Solomyak \cite{Birman_Solomyak}, G.\,Grubb \cite{Grubb_classic}, and M.\,Agranovich \cite{Agranovich}, which provide an analytic backbone for the related operator-theoretic constructions.

In all cases mentioned above, one can see the fundamental r\^{o}le of
a certain Herglotz operator-valued analytic function, which in
problems where a boundary is present (and sometimes even without an
explicit boundary \cite{Amrein-Pearson}) turns out to be a natural
generalisation of the classical notion of a Dirichlet-to-Neumann
map. The emergence of this object yields the possibility to apply to BVP advanced methods of complex analysis in conjunction with abstract methods of operator and spectral theory, which in turn sheds light
on the intrinsic interplay between the mentioned abstract frameworks
and concrete problems of interest in modern mathematical physics.


The present paper is a development of the recent activity \cite{KCher,KCherYulia,KCherYuliaNab,CherErKis} aimed at implementing the
above strategy in the context of problems of materials science and
wave propagation in inhomogeneous media. Our recent papers
\cite{ChKS_OTAA,CherednichenkoKiselevSilva} have shown that the language of
boundary triples is particularly fitting for direct and inverse
scattering problems on quantum graphs, as one of the key challenges to their analysis stems from the presence of interfaces through which
energy exchange between different components of the medium takes
place.  In the present work we continue the research initiated in
these papers, adapting the technology so that BVP, especially those
stemming from materials sciences, become within reach. As in
\cite{ChKS_OTAA,CherednichenkoKiselevSilva}, the ideas of \cite{Drogobych,MR573902} concerning the functional model
allow one to efficiently incorporate into the analysis information about the mentioned energy exchange, by employing a suitable Dirichlet-to-Neumann map.
In our analysis of BVP, 
we adopt the approach to the operator-theoretic treatment of BVP suggested by \cite{Ryzh_spec}, which appears to be particularly convenient for obtaining sharp quantitative information about scattering properties of the medium, {\it cf.} {\it e.g.} \cite{CherErKis}, where this same approach is used as a framework for the asymptotic analysis of homogenisation problems in resonant composites.

We next outline the structure of the paper. In Section
\ref{sec:triples-bvp} we recall the main points of the abstract
construction of \cite{Ryzh_spec}  and introduce the key objects for the analysis we carry out later on, such as the dissipative operator $L$ at the centre of the functional model. In Section \ref{sec:dilation-boundary-problem} we construct the minimal dilation of $L,$ based on the ideas of \cite{MR2330831}, which in the context of extensions of symmetric operators followed the earlier foundational work \cite{MR573902}. Using the functional model framework thus developed, in Section \ref{sec:functional-model} we construct a new version of Pavlov's ``three-component'' functional model for the dilation \cite{MR0510053} and pass to his ``two-component'', or ``symmetric", model \cite{Drogobych} (see also \cite{MR573902,MR2330831}), based on the notion of the characteristic function for $L,$ which is computed explicitly in terms of the $M$-operator introduced in Section \ref{sec:triples-bvp}. In Section \ref{resolvent_section} we develop formulae for the resolvents of  boundary-value operators for a range of boundary conditions $\alpha\Gamma_0u+\beta\Gamma_1u=0,$ with
$\alpha,$ $\beta$ from a wide class of operators in $L^2(\partial\Omega),$ including those relevant to applications.  The last two sections are devoted to the applications of the framework: based on the derived formulae for the resolvents, in Section \ref{model_sec} we establish the resolvent formulae for the operators of boundary-value problems belonging the class discussed earlier in the functional spaces stemming from the functional model, and in Section \ref{sec:example} we apply these formulae to obtain a description of the operators of BVPs in a class of Hilbert spaces with generating kernels. 

\section{Ryzhov triples for BVP}
\label{sec:triples-bvp}

In this section we follow \cite{Ryzh_spec} in developing an operator framework suitable for
dealing with boundary-value problems. The starting point is a
self-adjoint operator $A_0$ in a separable Hilbert space $\cH$ with
$0\in\rho(A_0)$, where $\rho(A_0),$ as usual, denotes the resolvent
set of $A_0$. Alongside  $\cH$, we consider an auxiliary Hilbert
space $\mathcal{E}$ and a bounded operator $\Pi: \mathcal{E}\to\mathcal{H}$
such that
\begin{equation}
  \label{eq:properties-a-pi}
  \dom(A_0)\cap\ran(\Pi)=\{0\}\quad\text{and}\quad \ker(\Pi)=\{0\}.
\end{equation}
Since $\Pi$ has a trivial kernel, there is a left inverse $\Pi^{-1},$ so that
$\Pi^{-1}\Pi=I_{\mathcal{E}}.$
We define
\begin{equation}
\label{eq:definition-A}
\begin{split}
  \dom(A)&:=\dom(A_0)\dotplus\ran(\Pi),\\
      A&:A_0^{-1}f+\Pi\phi\mapsto f,
\qquad f\in\mathcal{H}, \phi\in\mathcal{E},
\end{split}
\end{equation}
\begin{equation}
\label{eq:definition-G-0}
\begin{split}
  \dom(\Gamma_0)&:=\dom(A_0)\dotplus\ran(\Pi),\\
      \Gamma_0&:A_0^{-1}f+\Pi\phi\mapsto \phi,
\qquad f\in\mathcal{H}, \phi\in\mathcal{E},
\end{split}
\end{equation}
where neither $A$ nor $\Gamma_0$ is assumed closed or indeed closable.
The operator given in (\ref{eq:definition-A}) is the null extension of
$A_{0}$, while (\ref{eq:definition-G-0}) is the null extension of
$\Pi^{-1}$.
Note also that
\begin{equation}
 \label{eq:equality-domains-gamma-a}
  \ker(\Gamma_0)=\dom(A_0)\,.
\end{equation}
For $z\in\rho(A_0)$, consider the \emph{abstract} spectral boundary-value problem
\begin{equation}
 \label{eq:first-abstract-bv-problem}
  \begin{cases}
    Au = zu,\\[0.2em]
    \Gamma_0 u = \phi,\qquad \phi\in\mathcal{E},
  \end{cases}
\end{equation}
where the second equation is seen as a boundary
condition.  As asserted in \cite[Theorem 3.1]{Ryzh_spec}, there is
a unique solution $u$ of the boundary-value problem
(\ref{eq:first-abstract-bv-problem}) for any
$\phi\in\mathcal{E}$. Thus, there is an operator (clearly linear)
which assigns to any $\phi\in\mathcal{E}$ the solution $u$ of
(\ref{eq:first-abstract-bv-problem}), referred to as the
solution operator\footnote{The operator-valued function $\gamma$ is also sometimes referred to as the $\gamma$-field.} for $A$ and denoted by $\gamma(z).$ An explicit
expression for it in terms of $A_0$ and $\Pi$ can
be obtained as follows.  
Using the
fact that $A\supset A_0,$ one can show (see \cite[Remark 3.3]{Ryzh_spec}) that for all $\phi\in\mathcal{E}$ one has 
\[
\Pi\phi+z(A_0-zI)^{-1}\Pi\phi\in\ker(A-zI),\qquad\Gamma_0\bigl(\Pi\phi+z(A_0-zI)^{-1}\Pi\phi\bigr)=\phi,
\]
and therefore
\begin{equation}
  \label{eq:solution-operator}
  \gamma(z)\phi=(I+z(A_0-zI)^{-1})\Pi\phi\,.
\end{equation}

Furthermore, note that
\begin{equation}
  \label{eq:other-form-inverse}
  I+z(A_0-zI)^{-1}=(I-zA_0^{-1})^{-1},
\end{equation}
and so
(\ref{eq:definition-G-0}), (\ref{eq:solution-operator}) immediately
imply
\begin{equation}
 \label{eq:gamma-0-solution-operator}
  \Gamma_0\gamma(z)=I_{\mathcal{E}}\,.
\end{equation}
By (\ref{eq:solution-operator}),
one has
$\ran(\gamma(z))\subset\ker(A-zI)$, but the inverse inclusion also holds. Indeed,  taking a vector $u\in\ker(A-zI)$ and
writing it in the form $u=A_0^{-1}f +\Pi\phi$, one obtains
\begin{equation*}
  0=(A-zI)(A_0^{-1}f +\Pi\phi)=(I-zA_0^{-1})f-z\Pi\phi,
\end{equation*}
which yields $f=z(I-zA_0^{-1})^{-1}\Pi\phi$. Thus,
\begin{equation*}
  u=A_0^{-1}f
  +\Pi\phi=\left[zA_0^{-1}(I-zA_0^{-1})^{-1}+I\right]\Pi\phi
=(I-zA_0^{-1})^{-1}\Pi\phi\,.
\end{equation*}
In view of (\ref{eq:solution-operator}), (\ref{eq:other-form-inverse}), the last expression shows
that $u\in\ran(\gamma(z))$. Putting together the above, one arrives at
\begin{equation}
  \label{eq:range-solution-kernel-A}
  \ran(\gamma(z))=\ker(A-zI)\,.
\end{equation}
We remark that, since $A$ is not required to be closed,
$\ran(\gamma(z))$ is not necessarily a subspace. This is precisely the kind of situation that commonly occurs in the analysis of BVPs.

In what follows, we consider (abstract) BVP of the form
(\ref{eq:first-abstract-bv-problem}) associated with the operator $A$, with variable boundary conditions. To this
end, for a self-adjoint operator $\Lambda$ in $\mathcal{E},$ define
\begin{equation}
\label{eq:definition-G-kappa}
\begin{split}
  \dom(\Gamma_1)&:=\dom(A_0)\dotplus\Pi\dom(\Lambda),\\[0.3em]
  \Gamma_1&:A_0^{-1}f+\Pi\phi\mapsto\Pi^*f+\Lambda\phi,\qquad
  f\in\mathcal{H},\phi\in\dom(\Lambda).
\end{split}
\end{equation}
The operator $\Lambda$ can thus be seen as a parameter for the boundary operator $\Gamma_1.$

On the basis of \eqref{eq:solution-operator}, one obtains from
\eqref{eq:definition-G-kappa}  (see \cite[Equation 3.7]{Ryzh_spec}) that
\begin{equation}
  \label{eq:gamma-star-Gamma-1}
  \gamma(\cc{z})^*=\Gamma_1(A_0-zI)^{-1},\qquad z\in\rho(A_0).
\end{equation}
Also, according to \cite[Theorem 3.2]{Ryzh_spec}, the following Green's type identity holds:
\begin{equation}
  \label{eq:green-formula}
  \inner{Au}{v}_{\mathcal{H}}-\inner{u}{Av}_{\mathcal{H}}
=\inner{\Gamma_{1}u}{\Gamma_{0}v}_{\mathcal{E}}
-\inner{\Gamma_{0}u}{\Gamma_{1}v}_{\mathcal{E}},\qquad u,v\in\dom(\Gamma_{1})\,.
\end{equation}
The spectral BVP (\ref{eq:first-abstract-bv-problem}) is thus described by the triple
$(A_{0},\Pi,\Lambda)$, introduced by Ryzhov \cite{Ryzh_spec}.
His setup stems from the
Birman-Krein-Vishik theory
\cite{MR0080271,MR0024574,MR0024575,MR0051404}, rather than the theory
of boundary triples \cite{Gor}.
\begin{definition}
  \label{def:m-function}
  For a given triple $(A_{0},\Pi,\Lambda)$, define the operator-valued $M$-function associated with $A_0$ as follows: for any $z\in\rho(A_0)$,
  the operator $M(z)$ in $\mathcal{E}$  is defined on the domain $\dom(M(z)):=\dom(\Lambda),$ and its action is given by
  \begin{equation*}
    M(z): \phi\mapsto\Gamma_1\gamma(z)\phi,\qquad \phi\in{\rm dom}\bigl(M(z)\bigr).
  \end{equation*}
\end{definition}

The above abstract framework is illustrated (see \cite{Ryzh_spec} for details) by the
classical setup where $A_{0}$ is the Dirichlet Laplacian on a
bounded domain $\Omega$ with smooth boundary $\partial\Omega,$ so $A_0$ is self-adjoint on $\dom(A_0)=W^2_2(\Omega)\cap\mathring{W}^1_2(\Omega)$. In this
case $\Pi$ is simply the Poisson operator of harmonic lift, its left
inverse is the operator of boundary trace for harmonic functions and
$\Gamma_{0}$ is the null extension of the latter to
$\bigl[W^2_2(\Omega)\cap \mathring{W}^1_2(\Omega)\bigr]\dotplus\Pi
L^{2}(\partial\Omega)$. Furthermore, $\Lambda$ can be chosen as the
Dirichlet-to-Neumann map\footnote{
For convenience, we define the
  Dirichlet-to-Neumann map via $-\partial u/\partial n\vert_{\partial\Omega}$ instead of the more common $\partial u/\partial n\vert_{\partial\Omega}$. As a side note, we mention that this is obviously not the only choice for the operator $\Lambda.$ In particular, the trivial option $\Lambda=0$ is always possible. Our choice of $\Lambda$ is motivated by our interest in the analysis of classical boundary conditions.} which maps any function
$\phi\in W^1_2(\Omega)=:\dom(\Lambda)$ to
$-(\partial u/\partial n)\vert_{\partial\Omega}$, where $u$ is the
solution of the boundary-value problem
\begin{equation*}
  \begin{cases}
    \Delta u=0,\\
    u\vert_{\partial\Omega}=\phi,
  \end{cases}
\end{equation*}
(see {\it e.g.} \cite{Taylor_tools}). Due to the choice of $\Lambda$, it follows from (\ref{eq:definition-G-kappa}) that
\begin{equation}
  \label{eq:Gamma1-first-example}
   \dom(\Gamma_{1})=\bigl[ W^2_2(\Omega)\cap
     \mathring{W}^1_2(\Omega)\bigr]\dotplus\Pi W^1_2(\partial\Omega),\qquad
  \Gamma_{1}u=-\frac{\partial u}{\partial n}\biggr\vert_{\partial\Omega}.
\end{equation}
Note that (\ref{eq:Gamma1-first-example}) follows from the fact
that $\Pi^{*}f=-(\partial u/\partial n)\vert_{\partial\Omega}$ for
$u=A_{0}^{-1}f$. Therefore, the $M$-operator $M(z),$ $z\in\rho(A_{0}),$ is the
Dirichlet-to-Neumann map $\phi\mapsto-(\partial u/\partial n)\vert_{\partial\Omega}$ of the spectral boundary-value problem
(\ref{eq:first-abstract-bv-problem}), \ie $u\in\bigl[W^2_2(\Omega)\cap
\mathring{W}^1_2(\Omega)\bigr]\dotplus\Pi L^2(\partial\Omega)$
is a solution of
\begin{equation*}
  \begin{cases}
    \Delta u=zu,\\
    u\vert_{\partial\Omega}=\phi,
  \end{cases}
\end{equation*}
where $\phi$ belongs to $L^{2}(\partial\Omega),$ and $M(z)$ is understood as an unbounded operator\footnote{More precisely, $M(z)$ 
is the sum of an unbounded self-adjoint operator and a bounded one, which will be obvious from (\ref{M_operator}).} defined on $\dom(M(z))=W^1_2(\partial\Omega).$

This example shows how all the classical objects of BVP appear naturally from the triple $(A_{0},\Pi,\Lambda).$ In particular, it is worth noting how the energy-dependent Dirichlet-to-Neumann map $M(z)$ is ``grown'' from its ``germ'' $\Lambda$ at $z=0.$
Returning to the abstract setting and taking into account
(\ref{eq:definition-G-kappa}), one concludes from
Definition~\ref{def:m-function} that
\begin{equation}
\label{M_operator}
  M(z)=\Lambda + z\Pi^*(I-zA_0^{-1})^{-1}\Pi.
\end{equation}
From this equality, one verifies directly that
\begin{equation}
  \label{eq:difference-m-functions}
  \begin{split}
  M(z)-M(w)&=\Pi^*\left[z(I-zA_0^{-1})^{-1}-w(I-wA_0^{-1})^{-1}\right]\Pi
  =(z-w)\gamma(\cc{z})^*\gamma(w),\qquad z, w\in\rho(A_0).
\end{split}
\end{equation}
Also, due to the self-adjointness of
$\Lambda$, one has
\begin{equation}
 \label{eq:M-herglotz}
  M^*(z)=M(\cc{z}).
\end{equation}
The properties (\ref{eq:difference-m-functions}) and (\ref{eq:M-herglotz}) together imply that $M$ is an unbounded operator-valued Herglotz function,
\ie, $M(z)-M(0)$ is analytic, and $\Im M(z)\ge 0$ whenever
$z\in\complex_+$. It is shown in \cite[Theorem 3.3(4)]{Ryzh_spec} that
\begin{equation*}
  M(z)\Gamma_{0}u=\Gamma_{1}u\qquad\forall u\in\ker(A-zI)\cap\dom(\Gamma_{1}).
\end{equation*}
In this work we consider extensions (self-adjoint and non-selfadjoint) of the ``minimal'' operator
\begin{equation}
 \label{eq:B-definition}
  \widetilde{A}:=A_0\vert_{\ker(\Gamma_1)}
\end{equation}
that are restrictions of $A$.
It is proven in \cite[Section 5]{Ryzh_spec} that $\widetilde{A}$ is
symmetric with equal deficiency indices. Moreover,
\cite[Remark 5.1]{Ryzh_spec} asserts that
\begin{equation*}
  \dom(\widetilde{A})=A_0^{-1}[\ran(\Pi)^\perp],
\end{equation*}
so $\widetilde{A}$ does not depend on the parameter operator $\Lambda,$ contrary to what could be surmised from (\ref{eq:B-definition}).

Still following \cite{Ryzh_spec}, we let $\alpha$ and $\beta$ be linear operators in 
$\mathcal{E}$ such that $\dom(\alpha)\supset\dom(\Lambda)$
and $\beta$ is bounded on $\mathcal{E}$.
Additionally, assume that
$\alpha+\beta\Lambda$ is closable and denote its closure by $\ss.$
Consider
the linear set
\begin{equation}
  \label{eq:extension-of-boundary-conditions}
 {\mathcal H}_{\text \ss}:=\left\{A_{0}^{-1}f\dotplus\Pi\phi:f\in\mathcal{H},\,\,\phi
    \in\dom(\ss)\right\}.
\end{equation}
Following \cite[Lemma 4.1]{Ryzh_spec}, the identity
\[
(\alpha\Gamma_{0}+\beta\Gamma_{1})(A_0^{-1}f+\Pi\phi)=\beta\Pi^*f+(\alpha+\beta\Lambda)\phi,\qquad f\in {\mathcal H},\ \phi\in\dom(\Lambda),
\]
implies that $\alpha\Gamma_{0}+\beta\Gamma_{1}$ is well defined on
$\dom(A_{0})\dotplus\Pi\dom(\Lambda).$
The assumption that $\alpha+\beta\Lambda$ is closable is used to
extend the domain of definition of $\alpha\Gamma_{0}+\beta\Gamma_{1}$ to the set
\eqref{eq:extension-of-boundary-conditions}. Moreover, one verifies that
${\mathcal H}_{\text \ss}$
is a Hilbert space with respect to the norm
\[
\Vert u\Vert_{\ss}^2:=\Vert f\Vert_{\mathcal H}^2+\Vert\phi\Vert_{\mathcal E}^2+\Vert\ss\phi\Vert_{\mathcal E}^2,\quad u=A_0^{-1}f+\Pi\phi.
\]
It follows that the constructed extension $\alpha\Gamma_0+\beta\Gamma_1$ is a bounded operator from ${\mathcal H}_{\ss}$ to $\mathcal E.$

According to
\cite[Theorem 4.1]{Ryzh_spec}, if the operator
$\overline{\alpha+\beta M(z)}$ is boundedly invertible for
$z\in\rho(A_{0})$, the spectral boundary-value problem
\begin{equation}
 \label{eq:general-spectral-bv-problem}
    \begin{cases}
    (A- zI)u=f,\\[0.3em]
    \left(\alpha\Gamma_0+\beta\Gamma_{1}\right) u
    = \phi,\qquad f\in\mathcal{H},\,\,\phi\in\mathcal{E},
  \end{cases}
\end{equation}
has a unique solution $u\in{\mathcal H}_{\ss},$ where, as above, $\alpha\Gamma_0+\beta\Gamma_{1}$ is a bounded operator on
${\mathcal H}_{\ss}.$
Under the same
hypothesis of $\overline{\alpha+\beta M(z)}$ being
boundedly invertible for $z\in\rho(A_{0}),$ it follows from
\cite[Theorem 5.1]{Ryzh_spec} that the function
\begin{equation}
  (A_{0}-z I)^{-1}-(I-zA_{0}^{-1})^{-1}\Pi\bigl(\overline{\alpha+\beta M(z)}\bigr)^{-1}\beta\Pi^{*}
(I-zA_{0}^{-1})^{-1}
\label{Krein_formula1}
\end{equation}
is the resolvent of a closed operator
$A_{\alpha\beta}$ densely
defined in $\mathcal{H}$. Moreover,
$\widetilde{A}\subset A_{\alpha\beta}\subset A$ and
$\dom(A_{\alpha\beta})\subset\{u\in{\mathcal H}_{\ss}:
(\alpha\Gamma_{0}+\beta\Gamma_{1})u=0\}$.

Among the extensions $A_{\alpha\beta}$ of $\widetilde{A}$, we
single out the operator
\begin{equation}
\label{eq:operator-L}
    L:= A_{-\I I\,I},
\end{equation}
that is, $\alpha=-\I I$ and $\beta=I$.  Since in this
case $\alpha$ and $\beta$ are scalar operators, and
$\dom(\Gamma_{1})\subset\dom(\Gamma_{0})$, by virtue of (\ref{eq:extension-of-boundary-conditions}) one has
\begin{equation}
  \label{eq:inclusion}
  \dom(L)\subset\dom(\Gamma_{1})\,.
\end{equation}
The definition of $\dom(L)$ implies that for all
$h\in\mathcal{H},$ $z\in{\mathbb C}_-,$
\begin{align*}
  0&=(\Gamma_{1}-\I\Gamma_{0})(L-z I)^{-1}h
  =\Gamma_{1}(L-zI)^{-1}h-\I\Gamma_{0}[(L-zI)^{-1}-(A_{0}-zI)^{-1}]h\\[0.4em]
  &=M(z)\Gamma_{0}[(L-zI)^{-1}-(A_{0}-zI)^{-1}]h
    +\Gamma_{1}(A_{0}-zI)^{-1}h
    -\I\Gamma_{0}[(L-zI)^{-1}-(A_{0}-zI)^{-1}]h\\[0.4em]
  &=M(z)\Gamma_0(L-zI)^{-1}h+\Gamma_1(A_0-zI)^{-1}h-\I\Gamma_0(L-zI)^{-1}h,
\end{align*}
since, by \eqref{eq:equality-domains-gamma-a} and the fact that
$L,A_{0}\subset A,$ one has
\begin{equation*}
   \bigl[(L-zI)^{-1}-(A_{0}-zI)^{-1}\bigr]h\in\ker(A-z I),
   \qquad(A_{0}-zI)^{-1}h\in\ker(\Gamma_{0}).
\end{equation*}
Thus
\begin{equation}
  \label{eq:Gamma-0-Gamma-1-q}
 \begin{split}
   \Gamma_0(L-z I)^{-1}&=-(M(z)-\I I)^{-1}\Gamma_1(A_0-zI)^{-1},
   \quad z\in\complex_-,\\[0.3em]
   \Gamma_0(L^*-z I)^{-1}&=-(M(z)+\I I)^{-1}\Gamma_1(A_0-zI)^{-1},
   \quad z\in\complex_+,
\end{split}
\end{equation}
where the second equality is deduced in the same way as the first. In what follows, we will use the following relations, which are obtained by combining
(\ref{eq:gamma-star-Gamma-1}) and (\ref{eq:Gamma-0-Gamma-1-q}):
\begin{equation}
  \label{eq:Gamma-0-L-through-M-gamma}
\begin{split}
  \Gamma_0(L-z I)^{-1}&=-(M(z)-\I I)^{-1}\gamma(\cc{z})^*,
  \quad z\in\complex_-,\\[0.3em]
  \Gamma_0(L^*-z I)^{-1}&=-(M(z)+\I I)^{-1}\gamma(\cc{z})^*,
  \quad z\in\complex_+.
\end{split}
\end{equation}

It is proven in
\cite[Theorem 6.1]{Ryzh_spec} that the operator $L$ of formula (\ref{eq:operator-L}) is dissipative and boundedly invertible (hence maximal). We recall that
a densely defined operator $L$ in $\cH$
is called {\it dissipative} if
\begin{equation*}
  \im\inner{Lf}{f}\ge 0\quad \quad\forall f\in\dom(L).
\end{equation*}
A dissipative operator $L$ is said to be {\it maximal} if
$\complex_{-}\subset\rho(L)$.
Maximal dissipative
operators are closed, and any dissipative operator admits a maximal
extension. 

Furthermore, the function
\begin{equation}
  \label{eq:def-characteristic-of-L}
  S(z):=(M(z)-\I I)(M(z)+\I I)^{-1}
  =I-2\I(M(z)+\I I)^{-1},\qquad z\in\complex_+,
\end{equation}
turns out to be the characteristic function of $L,$ see \cite{livshitz, Strauss_survey}. Since $M$ is a Herglotz
function (see (\ref{eq:M-herglotz})), one has the following formula:
  \begin{equation}
 \label{eq:adjoint-characteristic-of-L}
 S^*(\cc{z}):=\left[S(\cc{z})\right]^{*}
 =I+2\I(M^{*}(\cc{z})-\I I)^{-1}=I+2\I(M(z)-\I I)^{-1},\qquad z\in\complex_-.
  \end{equation}
  We remark that the function $S$ is analytic in $\complex_+$ and, for each
  $z\in\complex_+$, the mapping $S(z):\mathcal{E}\to\mathcal{E}$ is a
  contraction. Therefore, $S$ has nontangential limits almost
  everywhere on the real line in the strong operator topology \cite{MR2760647}.

Recall that a closed
operator $L$ is said to be \emph{completely non-selfadjoint} if there
is no subspace reducing $L$ such that the part of $L$ in this subspace is self-adjoint. We refer to a completely non-selfadjoint symmetric operator 
as \emph{simple}.

\begin{proposition}
\label{prop:completely-nonselfadjoint}
If the symmetric operator $\widetilde{A}$ of (\ref{eq:B-definition}) is simple, then the
dissipative operator $L$ is completely non-selfadjoint.
  \end{proposition}
  \begin{proof}
    Suppose that $L$ has a reducing subspace $\mathcal{H}_{1}$ such that
    $L\vert_{\mathcal{H}_{1}}$ is self-adjoint. Take a nonzero
    $w\in\dom(L)\cap\mathcal{H}_{1}$. Then (\ref{eq:green-formula})
    and (\ref{eq:inclusion}) imply
    $  \inner{\Gamma_{1}w}{\Gamma_{0}w}_{\mathcal{E}}
-\inner{\Gamma_{0}w}{\Gamma_{1}w}_{\mathcal{E}}=0.
    $
Since $w\in\ker(\Gamma_{1}-\I\Gamma_{0})$, one obtains from the last
equality that $\norm{\Gamma_{0}w}=0$. Therefore,
$w\in\ker(\Gamma_{0})\cap\ker(\Gamma_{1})$, which means that
$w\in\dom(\widetilde{A})$.

The nontrivial invariant subspace $\mathcal{H}_{1}$ of $L$ is a
nontrivial invariant subspace of its restriction $\widetilde{A}$ as
long as $\mathcal{H}_{1}\cap\dom(\widetilde{A})\ne\emptyset$. This
last condition has been established above. Finally, since
$\widetilde{A}$ is symmetric, $\mathcal{H}_{1}$ is actually a reducing
subspace of $\widetilde{A}$. Clearly $\widetilde{A}$ is self-adjoint
in $\mathcal{H}_{1}$.
  \end{proof}


\section{Self-adjoint dilations for operators of BVP and a 3-component functional model}
\label{sec:dilation-boundary-problem}


Any completely non-selfadjoint dissipative operator $L$ admits a
self-adjoint dilation \cite{MR2760647}, which is unique up to a unitary transformation, under an assumption of minimality, see (\ref{eq:dilation-is-minimal}) below. There are numerous approaches
to an explicit construction of the named dilation
\cite{BMNW2018, MR573902, Naimark1940, Naimark1943, MR0510053, Drogobych, MR2330831, Ryzhov_closed, Strauss_survey}. In
applications, one is compelled to seek a realisation corresponding to
a particular setup. In the present paper we develop a way of constructing dilations
of dissipative operators convenient in the context of BVP for PDE.
 
 In the formulae below, we use the subscript ``$\pm$'' to
indicate two different versions of the same formula in which the
subscripts ``$+$'' and ``$-$'' are taken individually.

Recall that for any maximal dissipative operator $L,$ its {\it dilation} is defined as a self-adjoint operator $\mathscr{A}$ in a larger Hilbert space $\mathscr{H}\supset\mathcal{H}$
with the property
\begin{equation}
  \label{eq:extension-is-dilation}
 P_\mathcal{H}(\mathscr{A}-zI)^{-1}\big\vert_{\mathcal{H}}=
 (L-z I)^{-1}\, \qquad \forall z\in\complex_-.
\end{equation}
A dilation $\mathscr{A}$ is referred to as {\it minimal} if
\begin{equation}
  \label{eq:dilation-is-minimal}
  \overline{\Span_{z\in\complex\setminus\reals}\{(\mathscr{A}-z
  I)^{-1}\cH\}}
  =\mathscr{H}\,.
\end{equation}
We start by constructing a minimal dilation of the operator $L$ of the previous section, defined by (\ref{eq:operator-L}), following a procedure similar
to the one used in \cite{MR0365199,MR0510053}. Let
\begin{equation}
  \label{eq:wider-hilbert-pavlov}
  \mathscr{H}:=
  L^2(\reals_-,{\mathcal E})\oplus\mathcal{H}\oplus
  L^2(\reals_+,{\mathcal E})\,.
\end{equation}
In this Hilbert space, the operator
$\mathscr{A}$ is defined as follows. Its domain $\dom(\mathscr{A})$ is given by
\begin{equation}
  \dom(\mathscr{A}):=\left\{
   (v_-, u, v_+)^\top
\in\mathscr{H}:\ \ v_\pm\in W_2^1(\reals_\pm , \mathcal{E}),\ \ 
u\in\dom(\Gamma_1):\,
\begin{matrix}
\Gamma_1 u\pm\I\Gamma_0 u=\sqrt{2}v_\pm(0)
\end{matrix}
\right\},
\label{two_ast}
\end{equation}
where $W_2^1(\reals_+ , \mathcal{E})$ and
$W_2^1(\reals_- , \mathcal{E})$ are the Sobolev spaces
of
functions defined on $\reals_{+}$ and $\reals_{-}$, respectively, and
taking values in $\mathcal{E}$. We remark that the
results of the previous section imply that in our case
${\mathcal H}_{\ss}=\dom(\Gamma_1).$ On this domain, the operator
$\mathscr{A}$ acts according to the rule
\begin{equation}
\label{eq:dilation-action}
\mathscr{A}:
    \dom(\mathscr{A})\ni(v_-, u, v_+)^\top
\mapsto
(\I v'_-, Au, \I v'_+)^\top
    \in\dom(\mathscr{A}).
\end{equation}

\begin{theorem}
  \label{thm:dilation-selfadjoint}
  In the dilated space $\mathscr{H}$, the operator
  $\mathscr{A}$ is a self-adjoint extension of $L$.
\end{theorem}
\begin{proof}
  The fact that $\mathscr{A}$ is an extension of $L$ follows from
  \eqref{eq:operator-L} and \eqref{eq:inclusion}. Let us establish the
  self-adjointness of $\mathscr{A}$. Abbreviating
  $\mathfrak{u}=(v_-,u,v_+)^\top\in\dom(\mathscr{A}),$ we have
  \begin{equation}
  \label{first_starred}
  \begin{aligned}
    \inner{\mathscr{A}\mathfrak{u}}{\mathfrak{u}}-
    \inner{\mathfrak{u}}{\mathscr{A}\mathfrak{u}}&=
    \inner{\I v'_-}{v_-}+\inner{Au}{u}+\inner{\I v'_+}{v_+} -
    \inner{v_-}{\I v'_-}-\inner{u}{Au}-\inner{v_+}{\I v'_+}\\[0.3em]
   &=\I\int_{\reals_-}(v'_-\cc{v}_-+v_-\cc{v}_-') +
   \I\int_{\reals_+}(v'_+\cc{v}_++v_+\cc{v}_+')
   +\inner{Au}{u}-\inner{u}{Au}\\[0.2em]
   &=\I\norm{v_-(0)}^{2}-\I\norm{v_+(0)}^{2}+\inner{\Gamma_1
     u}{\Gamma_0 u}-\inner{\Gamma_0 u}{\Gamma_1 u}.
  \end{aligned}
  \end{equation}
Furthermore, taking into account the conditions defining $\dom(\mathscr{A})$,
one obtains
\begin{equation}
\label{second_starred}
\begin{aligned}
  \inner{\Gamma_1
     u}{\Gamma_0 u}&-\inner{\Gamma_0 u}{\Gamma_1 u}=
\bigl\langle\sqrt{2}v_-(0)+\I\Gamma_0u, \Gamma_0u\bigr\rangle-\bigl\langle\Gamma_0u, \sqrt{2}v_+(0)-\I\Gamma_0u\bigr\rangle\\[0.4em]
&=\bigl\langle\sqrt{2}v_{-}(0),\Gamma_{0}u\bigr\rangle-\bigl\langle\Gamma_{0}u, \sqrt{2}v_{+}(0)\bigr\rangle
                                                  =\I\bigl\langle v_-(0), v_+(0)-v_-(0)\bigr\rangle+\I\bigl\langle v_+(0)-v_-(0), v_+(0)\bigr\rangle
                                                  \\[0.4em]
&
=-\I\norm{v_-(0)}^{2}+\I\norm{v_+(0)}^{2}\,.
\end{aligned}
\end{equation}
It follows by combining (\ref{first_starred})  and (\ref{second_starred}) that $\mathscr{A}$ is symmetric. To complete the proof, it suffices to
show that $\ran(\mathscr{A}-z I)=\mathscr{H}$ for all $z\in\complex\setminus\reals$. To this end, consider the operators $\partial_\pm$ and $\partial_\pm^0$ in $L_2(\reals_\pm,\mathcal{E})$ given by
\begin{equation*}
  \dom(\partial_\pm):=W_2^1(\reals_\pm,\mathcal{E}),\qquad \partial_\pm:y_\pm\mapsto iy'_\pm,
\qquad \qquad \dom(\partial_\pm^0):=\stackrel{\hspace{-1mm}\circ}{W_2^1}(\reals_\pm,\mathcal{E}), \qquad \partial_\pm^0:y_\pm\mapsto iy'_\pm.
\end{equation*}
Here,
$\overset{\hspace{-1mm}\circ}{W_{2}^{1}}(\reals_\pm,\mathcal{E})$ is
the closure in $W_2^1(\reals_\pm,\mathcal{E})$ of the set of smooth
functions with compact suppport in $\reals_{\pm}.$
The operators
$\partial_{+}^{0}$ and $\partial_{-}^{0}$ are symmetric, with
deficiency indices $(n_+,n_-)=(1,0)$ and $(n_+,n_-)=(0,1)$, respectively. Also,
$\partial_{\pm}^{0}=\partial_{\pm}^{*}$ (see \cite[Chapter 4,
Section 8.4]{MR1192782}). Therefore $\rho(\partial_\pm)=\complex_\pm$ and
$\rho(\partial_\pm^0)=\complex_\mp$.

Take any $z\in\complex_-$ and $(h_-,h,h_+)^\top\in\mathscr{H}$. It turns
out that the vector
$(f_-,f,f_+)^\top$ defined
by
\begin{equation}
\label{eq:vector-in-domain}
\begin{split}
    f_-&:=(\partial_--zI)^{-1}h_-,\\[0.25em]
    f&:=(L -z I)^{-1}h+\sqrt{2}\gamma(z)(M(z)-\I I)^{-1}f_-(0),\\[0.3em]
    f_+&:=(\partial_+^0-zI)^{-1}h_+
    +{\rm e}^{-\I z\cdot}\bigl[\I\sqrt{2}\Gamma_0(L-zI)^{-1}h
      + S^*(\cc{z})f_-(0)\bigr],
\end{split}
\end{equation}
is an element of $\dom(\mathscr{A})$. Indeed, clearly $f_{-}\in W_{2}^{1}(\reals_{-}, \mathcal{E}),$ and $f_+ \in W_{2}^{1}(\reals_{+}, \mathcal{E})$ since 
\begin{equation}
f_+=(\partial_+^0-zI)^{-1}h_+
    +{\rm e}^{-\I z\cdot}\mathbf{e}\quad {\rm for\ some}\ \mathbf{e}\in\mathcal{E}.
    \label{new_star}
\end{equation}
 Also,
\begin{align*}
  (\Gamma_1-\I\Gamma_0)\bigl\{(L -z I)^{-1}h
    +\sqrt{2}\gamma(z)(M(z)-\I I)^{-1}f_-(0)\bigr\}
&=(\Gamma_1-\I\Gamma_0)\sqrt{2}\gamma(z)(M(z)-\I I)^{-1}f_-(0)\\[0.3em]
&=\sqrt{2}(M(z)-\I I)(M(z)-\I I)^{-1}f_-(0)
=\sqrt{2}f_-(0),
\end{align*}
where to obtain the first equality we use \eqref{eq:operator-L}, and the second equality follows from (\ref{eq:gamma-0-solution-operator}) and Definition~\ref{def:m-function}.
Thus
\begin{equation}
  \label{eq:first-domain-condition}
  (\Gamma_1-\I\Gamma_0)f=\sqrt{2}f_-(0)\,.
\end{equation}
In addition, we have
\begin{align*}
  (\Gamma_1+\I\Gamma_0)f&=(\Gamma_1-\I\Gamma_0)f +2\I\Gamma_0f
=\sqrt{2}f_-(0) + 2\I\Gamma_0\bigl\{(L -z I)^{-1}h
    +\sqrt{2}\gamma(z)(M(z)-\I I)^{-1}f_-(0)\bigr\}\\[0.4em]
&=\sqrt{2}f_-(0)+2\I \Gamma_0(L -z I)^{-1}h+\I 2\sqrt{2}(M(z)-\I I)^{-1}f_-(0)\\[0.4em]
&=2\I\Gamma_0(L -z I)^{-1}h+\sqrt{2}\left[I+2\I (M(z)-\I I)^{-1}\right]f_-(0)
=2\I\Gamma_0(L -z I)^{-1}h+\sqrt{2}S^*(\cc{z})f_-(0),
\end{align*}
where we have used (\ref{eq:first-domain-condition}), (\ref{eq:vector-in-domain}) for the
second, (\ref{eq:gamma-0-solution-operator}) for the third,
and (\ref{eq:adjoint-characteristic-of-L}) for the fifth
equality. Due to the expression for $f_+$ in (\ref{eq:vector-in-domain}), we have thus shown that
\begin{equation}
  \label{eq:second-domain-condition}
  (\Gamma_1+\I\Gamma_0)f=\sqrt{2}f_+(0)\,.
\end{equation}
The equalities (\ref{eq:first-domain-condition}) and 
(\ref{eq:second-domain-condition}) imply that $(f_-,f,f_+)^\top\in\dom(\mathscr{A}),$ see (\ref{two_ast}).

Next, we show that
\begin{equation}
  \label{eq:dilation-z-on-f}
  (\mathscr{A}-zI)
  \begin{pmatrix}
    f_-\\ f\\ f_+
  \end{pmatrix}=
  \begin{pmatrix}
    h_-\\ h\\ h_+
  \end{pmatrix}\,.
\end{equation}
 On the one hand, it follows from (\ref{new_star}) and the first line of (\ref{eq:vector-in-domain}) that
\begin{equation}
 \label{eq:h-pm-f-pm}
  h_\pm=(\partial_\pm-zI)f_\pm.
\end{equation}
On the
other hand, due to the fact that $L\subset A$ and the property (\ref{eq:range-solution-kernel-A}), one has
\begin{equation}
  \label{eq:A-z-acting-on-f}
  (A-zI)\bigl[(L -z I)^{-1}h
    +\sqrt{2}\gamma(z)(M(z)-\I I)^{-1}f_-(0)\bigr]=h\,.
\end{equation}
In conformity with (\ref{eq:dilation-action}), the identities
(\ref{eq:h-pm-f-pm}), (\ref{eq:A-z-acting-on-f}) yield
(\ref{eq:dilation-z-on-f}). As $(h_-,h,h_+)^\top$ is an arbitrary element in
$\mathscr{H}$. we have also shown that
$\ran(\mathscr{A}-zI)=\mathscr{H}$ for $z\in\complex_-$.

Now fix an arbitrary $z\in\complex_+$. For any
$(h_-,h,h_+)^\top\in\mathscr{H}$, we redefine
\begin{equation}
  \label{eq:other-vector-in-domain}
  \begin{split}
    f_+&:=(\partial_+-zI)^{-1}h_+,\\[0.3em]
    f&:=(L^{*} -z I)^{-1}h
    +\sqrt{2}\gamma(z)(M(z)+\I I)^{-1}f_+(0),\\[0.4em]
      f_-&:=(\partial_-^0-zI)^{-1}h_-+
    {\rm e}^{-{\rm i}z\cdot}\bigl[-\I\sqrt{2}\Gamma_0(L^{*}-zI)^{-1}h+
      S(z)f_+(0)\bigr]\,.
\end{split}
\end{equation}
In the
same way as above, it can be shown that
$(f_-,f,f_+)^\top\in\mathscr{A}$ and
\begin{equation*}
  (\mathscr{A}-zI)
  \begin{pmatrix}
    f_-\\ f\\ f_+
  \end{pmatrix}=
  \begin{pmatrix}
    h_-\\ h\\ h_+
  \end{pmatrix},
\end{equation*}
which completes the proof.
\end{proof}
\begin{remark}
  \label{rem:resolvent-formulae}
  In the proof of Theorem~\ref{thm:dilation-selfadjoint}, we have
  obtained the following formulae for the resolvent of
  $\mathscr{A}:$ for $(h_-,h,h_+)^\top\in\mathscr{H}$, 
  \begin{equation*}
    (\mathscr{A}-zI)^{-1}
     \begin{pmatrix}
    h_-\\ h\\ h_+
  \end{pmatrix}=\begin{pmatrix}
    f_-\\ f\\ f_+
  \end{pmatrix},
  \end{equation*}
 where $(f_-,f,f_+)^\top$ is given by \eqref{eq:vector-in-domain}
  for $z\in\complex_-$ and by
  \eqref{eq:other-vector-in-domain} for $z\in\complex_+$.
\end{remark}

The following technical result will be used to prove that $\mathscr{A}$ is a minimal dilation of $L$; at the same time, it is of a clear independent interest.
\begin{lemma}
  \label{lem:expected-density}
 Each of the sets
 \[ \Span_{u\in\mathcal{H}}\left\{\Gamma_{1}(A_{0}-zI)^{-1}u\right\},\qquad
    \Span_{    h\in\mathcal{H}}\left\{\Gamma_{0}(L-zI)^{-1}h\right\},\qquad
    \Span_{    h\in\mathcal{H}}\left\{\Gamma_{0}(L^*-zI)^{-1}h\right\}
     \]
is dense in $\mathcal{E},$ for every $z\in \mathbb{C}_-\cup \mathbb{C}_+,$ $z\in \mathbb{C}_-,$ $z\in \mathbb{C}_+,$ respectively.
\end{lemma}
\begin{proof}
  Due to \eqref{eq:Gamma-0-Gamma-1-q} and the fact that $\dom(M(z))$ is
  dense in $\mathcal{E}$, it suffices to prove the assertion of the lemma about the first set.

  Suppose that $\widetilde{v}\in\mathcal{E}$ is such that
   $ \inner{\Gamma_{1}(A_{0}-z I)^{-1}u}{\widetilde{v}}=0$ for all $u\in{\mathcal H}.$
  Using \eqref{eq:gamma-star-Gamma-1}, 
  we obtain
  $
  \inner{u}{\gamma(\bar z)\widetilde{v}}=0\quad \forall u\in \mathcal H,
  $
  and therefore
  $\gamma(\bar z)\widetilde{v}=0$ or, in view of (\ref{eq:solution-operator}), $v=-\cc{z}(A_0-\cc{z})^{-1}v,$
  where $v:=\Pi\widetilde{v}$. Hence, $v\in \dom A_0$ and $\widetilde{v}=\Gamma_0 v=0,$ as required.
\end{proof}
\begin{theorem}
  \label{thm:minimal-dilation}
  The operator $\mathscr{A}$ is a minimal
  self-adjoint dilation of $L$.
\end{theorem}
\begin{proof}
  By Theorem~\ref{thm:dilation-selfadjoint}, the operator
  $\mathscr{A}$ is a self-adjoint extension of $L$. The property
  \eqref{eq:extension-is-dilation} is verified directly on the basis of Remark \ref{rem:resolvent-formulae}. Thus it only remains to check the minimality condition \eqref{eq:dilation-is-minimal}.
  It follows from Remark~\ref{rem:resolvent-formulae} that, relative
  to the orthogonal decomposition
  \eqref{eq:wider-hilbert-pavlov}, one has
\begin{align*}
  \Span_{z\in\complex\setminus\reals}\{(\mathscr{A}-z I)^{-1}\cH\}
  &=\Span_{z\in\complex_{+}}\{{\rm e}^{-\I z\cdot}\Gamma_0(L^{*}-zI)^{-1}\cH\}\\[0.1em]
  &\oplus\biggl(\Span_{z\in\complex_{-}}\{(L-z I)^{-1}\cH\}+
    \Span_{z\in\complex_{+}}\{(L^{*}-z I)^{-1}\cH\}\biggr)
  \oplus\Span_{z\in\complex_{-}}\{{\rm e}^{-\I z\cdot}\Gamma_0(L-zI)^{-1}\cH\}\,.
\end{align*}
Since $L$ is densely defined, one clearly has
\begin{equation*}
  \overline{\Span_{z\in\complex_{-}}\{(L-z I)^{-1}\cH\}+
    \Span_{z\in\complex_{+}}\{(L^{*}-z I)^{-1}\cH\}}=\cH.
\end{equation*}
We next show that
\begin{equation}
  \label{eq:prospect-for-dense-set}
  \overline{\Span_{z\in\complex_{-}}
    \{{\rm e}^{-\I z\cdot}\Gamma_0(L-zI)^{-1}\cH\}}=
  L_{2}(\reals_{+},\mathcal{E})\,.
\end{equation}
Assuming that
$g\in L_{2}(\reals_{+},\mathcal{E})$ is such that for all $z\in{\mathbb C}_-,$ $h\in{\mathcal H},$ one has
\begin{align*}
  0=\inner{{\rm e}^{-\I z\cdot}\Gamma_0(L-zI)^{-1}h}{g}_{L_{2}(\reals_{+},\mathcal{E})}
   &=\int_{\reals_{+}}\inner{\Gamma_0(L-zI)^{-1}h}{{\rm e}^{\I\overline{z}\xi}g(\xi)}_{\mathcal{E}}d\xi\\[0.4em]
    &=\inner{\Gamma_0(L-zI)^{-1}h}{\int_{\reals_{+}}{\rm e}^{\I \xi\re z}{\rm e}^{\xi\im z}g(\xi)d\xi}_{\mathcal{E}}.
\end{align*}
By Lemma \ref{lem:expected-density}, it follows that
\begin{equation*}
\int_{\reals_{+}}{\rm e}^{\I \xi\re z}{\rm e}^{\xi\im z}g(\xi)d\xi=0\qquad\forall z\in{\mathbb C}_-.
\end{equation*}
Finally, fixing $\im z$ and taking the Fourier transform with respect to $\re z$ yields $g(\xi)=0$ for a.e. $\xi\in\reals_{+}$, which concludes the proof of \eqref{eq:prospect-for-dense-set}. By a similar argument, one also shows that
\begin{equation*}
  \overline{\Span_{z\in\complex_{+}}
    \{{\rm e}^{-\I z\cdot}\Gamma_0(L^{*}-zI)^{-1}\cH\}}=
  L_{2}(\reals_{-},\mathcal{E}),
\end{equation*}
which completes the proof.
\end{proof}
For convenience, we introduce the following families of sets in $\mathscr{H}$. For any $z_+\in\complex_+$ and $z_-\in\complex_-$, define
\begin{align*}
  \mathcal{Y}(z_{+},z_{-})&:=(\mathscr{A}-z_+ I)^{-1}
      \begin{pmatrix}
      0\\0\\
      L_2(\reals_+,\mathcal{E})
    \end{pmatrix}
    +(\mathscr{A}-z_- I)^{-1}
    \begin{pmatrix}
L_2(\reals_-,\mathcal{E})\\0\\0
\end{pmatrix},\\[0.7em]
  \mathcal{G}(z_+,z_-)&:=P_{\mathcal{H}}\mathcal{Y}(z_{+},z_{-}),
\end{align*}
where $P_{\mathcal{H}}$ is the orthogonal projection onto $\{0\}\oplus\mathcal{H}\oplus\{0\}.$ Henceforth, we identify $\{0\}\oplus\mathcal{H}\oplus\{0\}$ and $\mathcal{H}.$
\begin{lemma}
  \label{lem:set-resolvent-dense-whole-space}
  If $\widetilde{A}$ defined in \eqref{eq:B-definition} is simple, then the linear sets
  \begin{equation}
  \label{eq:set-dense-in-pavlov}
\Span_{z_\pm\in{\mathbb C}_\pm}
    \mathcal{Y}(z_{+},z_{-}),\qquad
    \Span_{z_\pm\in{\mathbb C}_\pm}\mathcal{G}(z_+,z_-)
\end{equation}
are dense in the
spaces $\mathscr{H}$ and $\mathcal{H}$, respectively.
\end{lemma}
\begin{proof}
  To simplify notation, denote by $Y$ the closure of the first set in \eqref{eq:set-dense-in-pavlov}. It follows from
  Remark~\ref{rem:resolvent-formulae} that
  \begin{equation}
    \label{eq:first-second-inclusion}
    \{0\}\oplus\{0\}\oplus L_2(\reals_+,\mathcal{E})\subset Y,\qquad 
     L_2(\reals_-,\mathcal{E})\oplus\{0\}\oplus\{0\}\subset Y\,.
  \end{equation}
  Indeed, 
  putting $h_{-}=0\in L_2(\reals_-,\mathcal{E}),$ $h=0\in\mathcal{H}$ in
  \eqref{eq:other-vector-in-domain} with $z=z_+\in{\mathbb C}_+\subset\rho(\partial_{+})$, the first inclusion in (\ref{eq:first-second-inclusion}) follows.
  Similarly, by putting $h_{+}=0\in L_2(\reals_+,\mathcal{E}),$ $h=0\in\mathcal{H}$ in \eqref{eq:vector-in-domain} with $z=z_-\in{\mathbb C}_-\subset\rho(\partial_{-})$, the second inclusion in
  \eqref{eq:first-second-inclusion} follows.
  The inclusions (\ref{eq:first-second-inclusion}) imply that the
  orthogonal complement of $Y$ is a subset of $\cH$.

 It remains to show that
  \begin{equation}
    \overline{\Span_{z_\pm\in{\mathbb C}_\pm}\mathcal{G}(z_+,z_-)}=\mathcal{H}\,.
        \label{dot_eq}
\end{equation}
Using the formulae for the resolvent of the dilation (see \eqref{eq:vector-in-domain}
  for $z\in\complex_-,$ \eqref{eq:other-vector-in-domain} for $z\in\complex_+,$ and
Remark~\ref{rem:resolvent-formulae}), one immediately obtains 
\begin{equation}
  \label{eq:g-definition}
  \mathcal{G}(z_+,z_-)=\gamma(z_+)(M(z_+)+\I
    I)^{-1}\mathcal{E}+\gamma(z_-)(M(z_-)-\I I)^{-1}\mathcal{E}\,.
  \end{equation}
  Suppose that  $u\in \mathcal H$ is such that $u\perp G(z_+,z_-)$ for all $z_+\in{\mathbb C}_+,$ $z_-\in{\mathbb C}_-.$ 
Taking into account that vectors in $\mathcal E$ in \eqref{eq:g-definition} can be chosen independently in the first and second summands, 
 we obtain
  \begin{equation*}
    \inner{\gamma(z_+)(M(z_+)+\I
    I)^{-1}\mathbf{e}}{u}=\inner{\gamma(z_-)(M(z_-)-\I I)^{-1}\mathbf{e}}{u}=0\qquad\forall z_+\in{\mathbb C}_+,\ z_-\in{\mathbb C}_-,\ \mathbf{e}\in\mathcal{E}.
  \end{equation*}
In particular, for $z_+\in{\mathbb C}_+$ we have
  \begin{align*}
    0&=\inner{\gamma(z_+)(M(z_+)+\I I)^{-1}\mathbf{e}}{u}\\[0.3em]
      &=\inner{(M(z_+)+\I
                            I)^{-1}\mathbf{e}}{\gamma(z_+)^*u}
                            =\inner{(M(z_+)+\I
                            I)^{-1}\mathbf{e}}{\Gamma_1(A_0-\cc{z}_+)^{-1}u}\qquad\forall \mathbf{e}\in\mathcal{E}.
  \end{align*}
  Since $\overline{(M(z_+)+\I I)^{-1}\mathcal{E}}=\mathcal{E}$, it follows that $\Gamma_1(A_0-\cc{z}_+I)^{-1}u=0,$ and hence $(A_0-\cc{z}_+I)^{-1}u\in\dom(\widetilde{A})$. Finally, noticing that $(\widetilde{A}-\cc{z}_+I)(A_0-\cc{z}_+I)^{-1}u=u,$ we conclude that $u\in\ran(\widetilde{A}-\cc{z}_+I).$ Similarly, we establish that $u\in\ran(\widetilde{A}-\cc{z}_- I)$ for $z_-\in{\mathbb C}_-.$ Since $z_+\in\complex_+,$ $z_-\in\complex_-$ above are arbitrary, it follows that
  \begin{equation}
    \label{eq:inclusion-in-trivial}
    u\in\bigcap_{z\in\complex\setminus\reals}\ran(\widetilde{A}-zI)\,.
  \end{equation}
  The assumption that $\widetilde{A}$ is simple is equivalent (see \cite[Section 1.3]{MR0048704}) to the fact that the set on the right-hand side of \eqref{eq:inclusion-in-trivial} is trivial, and hence $u=0.$ This concludes the proof of (\ref{dot_eq}).\end{proof}


\begin{remark}
  The terms on the right-hand side of \eqref{eq:g-definition} are
  linearly independent.
\end{remark}

\begin{proof}
Assume that $\mathbf{e}_1, \mathbf{e}_2\in{\mathcal E}$ are such that
\begin{equation}
\gamma(z_+)\bigl(M(z_+)+\I I\bigr)^{-1}\mathbf{e}_1+\gamma(z_-)\bigl(M(z_-)-\I I\bigr)^{-1}\mathbf{e}_2=0.
\label{lin_comb}
\end{equation}
Applying $\Gamma_0$ and $\Gamma_1$ to (\ref{lin_comb}) and using the definition of $\gamma,$
we obtain
\begin{equation}\label{eq:idiocy}
\begin{gathered}
\bigl(M(z_+)+\I I\bigr)^{-1}\mathbf{e}_1+\bigl(M(z_-)-\I I\bigr)^{-1}\mathbf{e}_2=0,\\[0.3em]
  \mathbf{e}_1-\I\bigl(M(z_+)+\I I\bigr)^{-1}\mathbf{e}_1+\mathbf{e}_2+\I\bigl(M(z_-)-\I I\bigr)^{-1}\mathbf{e}_2=0,
\end{gathered}
\end{equation}
respectively. Substituting the first identity above into the second one yields
$$
\mathbf{e}_2+\bigl(M(z_+)-\I I\bigr)\bigl(M(z_+)+\I I\bigr)^{-1}\mathbf{e}_1=0.
$$
Then the first equality in \eqref{eq:idiocy} becomes
$(M(z_-)-\I I)^{-1}(M(z_+)-\I I)\mathbf{f}_1=\mathbf{f}_1$, where the substitution $\mathbf{f}_1:=(M(z_+)+\I I)^{-1}\mathbf{e}_1$ has been used.
It follows that $M(z_-)\mathbf{f}_1=M(z_+)\mathbf{f}_1$. Setting $z_-=\overline{z}_+,$ 
we obtain, in particular,
$\Im \inner{M(z_+)\mathbf{f}_1}{\mathbf{f}_1}=0.$ Combined with the property $\Im M(z_+)=(\Im z_+) \gamma^*(z_+)\gamma(z_+)$ (see (\ref{eq:difference-m-functions})), this implies $\|\gamma(z_+)\mathbf{f}_1\|=0$, which immediately leads to $\mathbf{f}_1=0$ due to \eqref{eq:properties-a-pi}. Finally, we infer $\mathbf{e}_1=(M(z_+)+\I I)\mathbf{f}_1=0$, as required.
\end{proof}

\section{Two-component spectral form of the functional model}
\label{sec:functional-model}

Following \cite{MR573902}, we introduce a Hilbert space in which we construct
a functional model for the operator family $A_{\alpha\beta},$ in the spirit of Pavlov  \cite{MR0365199, MR0510053, Drogobych}.
The functional model for completely
non-selfadjoint maximal dissipative operators that can be represented as additive perturbations of self-adjoint operators was constructed in \cite{MR0365199, MR0510053, Drogobych} and further developed in
\cite{MR573902} to include non-dissipative operators. In the context of boundary triples an analogous construction was carried out in \cite{MR2330831}. In the most general setting to date, namely the setting of adjoint operator pairs, an {\it explicit} three-component model akin to the one we presented in the previous section was constructed in \cite{BMNW2018}, which however stops short of constructing a ``spectral", two-component, form of the model, which is particularly convenient for the development of 
a scattering theory for operator pairs.\footnote{We refer the reader to the paper
\cite{Ryzhov_closed}, where a three-component model is constructed for a dissipative operator with at least one regular point in the upper half-plane.} In this section we 
we carry out such a construction, tailored to study operators of BVP, in the case when symbol of the operator is formally self-adjoint (but the operator itself can be non-selfadjoint due to the boundary conditions). 

Next, we recall some concepts relevant to the
construction of \cite{MR573902}. In what follows, we assume throughout that $\widetilde{A},$ see (\ref{eq:B-definition}), is simple and therefore $L$ is completely non-selfadjoint (see Proposition \ref{prop:completely-nonselfadjoint}).

A function $f,$ analytic on $\complex_\pm$ and taking values in $\mathcal{E},$ is said to be
in the Hardy class $H^2_\pm(\mathcal{E})$ when
\begin{equation*}
  \sup_{y>0}\int_\reals\bigl\Vert f(x\pm\I y)\bigr\Vert_{\mathcal{E}}^2dx<+\infty
\end{equation*}
({\it cf.} \cite[Sec.\,4.8]{MR822228}). If
$f\in H^2_\pm(\mathcal{E})$, then the left-hand side of the above
inequality defines $\norm{f}_{H^2_\pm(\mathcal{E})}^2.$ 

Any element in $H^2_\pm(\mathcal{E})$ can be associated with its
boundary values existing almost everywhere on the real line. It will
cause no confusion if we use the same notation,
$H_\pm^2(\mathcal{E})$, to denote the spaces of boundary functions. By
\cite[Sec.\,4.8, Thm.\,B]{MR822228}, $H_\pm^2(\mathcal{E})$ are
subspaces of $L^2(\reals,\mathcal{E})$. Also, due to the Paley-Wiener theorem
\cite[Sec.\,4.8, Thm.\,E]{MR822228}), one verifies that these
subspaces are the orthogonal complements of each other ({\it i.e.},
$L^2(\reals,\mathcal{E})=H^2_+(\mathcal{E})\oplus H^2_-(\mathcal{E})$).

We now return to the setup of Section \ref{sec:triples-bvp} and prove a fundamental regularity property for the expressions (\ref{eq:Gamma-0-L-through-M-gamma}), which is crucial for our construction.
\begin{lemma}
  \label{lem:bounded-hardy}
  Let the operators $\Gamma_0$ and $L$ be defined by (\ref{eq:definition-G-0}) and (\ref{eq:operator-L}), respectively. For all $h\in\mathcal{H}$, one has $\Gamma_0(L-\cdot)^{-1}h\in H_2^-(\mathcal{E})$ and $\Gamma_0(L^*-\cdot)^{-1}h\in H_2^+(\mathcal{E})$. Moreover,
  \begin{equation}
    \norm{\Gamma_0(L-\cdot)^{-1}h}_{H_2^-(\mathcal{E})}\le\sqrt{\pi}\norm{h}_{\mathcal{H}},\qquad
    \norm{\Gamma_0(L^*-\cdot)^{-1}h}_{H_2^+(\mathcal{E})}\le\sqrt{\pi}\norm{h}_{\mathcal{H}}\,.
    \label{trian}
  \end{equation}
\end{lemma}
\begin{proof}
  The resoning goes along the lines of the proof of
  \cite[Lem.\,2.4]{MR2330831} which in turn is based on the one of
  \cite[Thm.\,1]{MR573902}.

  Suppose that $z\in{\mathbb C}_-.$ Using the Green's identity \eqref{eq:green-formula} and the fact that
  $L\subset A,$ we obtain, for all $h\in{\mathcal H},$ 
    \begin{align*}
      2\I\norm{\Gamma_0(L-zI)^{-1}h}^2
         &=\inner{\I\Gamma_0(L-zI)^{-1}h}{\Gamma_0(L-zI)^{-1}h}
           -\inner{\Gamma_0(L-zI)^{-1}h}{\I\Gamma_0(L-zI)^{-1}h}\\[0.4em]
         &=\inner{\Gamma_1(L-zI)^{-1}h}{\Gamma_0(L-zI)^{-1}h}
           -\inner{\Gamma_0(L-zI)^{-1}h}{\Gamma_1(L-zI)^{-1}h}\\[0.4em]
         &=\inner{L(L-zI)^{-1}h}{(L-zI)^{-1}h}
           -\inner{(L-zI)^{-1}h}{L(L-zI)^{-1}h}\\[0.4em]
         &=\inner{h}{(L-zI)^{-1}h}-\inner{(L-zI)^{-1}h}{h}
           +(z-\cc{z})\norm{(L-zI)^{-1}h}^2\,.
    \end{align*}
    Since $L$ is maximal dissipative, it admits a self-adjoint dilation $\mathscr{A}$ \cite{MR2760647}.  (In the case of the operator $L$ considered here, this dilation is given explicitly by Theorem \ref{thm:minimal-dilation}. However, we do not require this fact here.)
    One concludes, by resorting to the resolvent identity,  that
    \begin{align*}
      \norm{\Gamma_0(L-zI)^{-1}h}^2
      &=\frac{1}{2\I}\left\{\inner{\bigl[(\mathscr{A}-\cc{z}I)^{-1}
      -(\mathscr{A}-zI)^{-1}\bigr]h}{h}+(z-\cc{z})\norm{(L-zI)^{-1}h}^2\right\}\\[0.35em]
      &=\frac{1}{2\I}\left\{(\cc{z}-z)\norm{(\mathscr{A}-zI)^{-1}h}^{2}
        +(z-\cc{z})\norm{(L-zI)^{-1}h}^2\right\}\,.
    \end{align*}
    Denoting by $E(t)$, $t\in\reals,$ the resolution of identity \cite[Chapter 6]{MR1192782} for $\mathscr{A}$ and setting $z=k-\I\epsilon$, $k\in\reals,$ $\epsilon>0,$ one has
    \begin{equation*}
      \norm{\Gamma_0(L-(k-\I\epsilon)I)^{-1}h}^2
      =\epsilon\int_{\reals}\frac{d\inner{E(t)h}{h}}{(t-k)^{2}+
        \epsilon^{2}}-\epsilon\norm{(L-(k-\I\epsilon)I)^{-1}h}^2\,.
    \end{equation*}
    Now, using Fubini's theorem, we obtain
    \begin{align*}
      \int_{\reals}\norm{\Gamma_0(L-(k-\I\epsilon)I)^{-1}h}^2dk&=\int_{\reals}\left(\epsilon\int_{\reals}\frac{d\inner{E(t)h}{h}}
      {(t-k)^{2}+\epsilon^{2}}\right)dk-\epsilon\int_{\reals}
      \norm{(L-(k-\I\epsilon)I)^{-1}h}^{2}dk\\[0.6em]
      &=\int_{\reals}\left(\epsilon\int_{\reals}\frac{dk}
      {(t-k)^{2}+\epsilon^{2}}\right)d\inner{E(t)h}{h}-\epsilon\int_{\reals}
        \norm{(L-(k-\I\epsilon)I)^{-1}h}^{2}dk\\[0.6em]
      &=\pi\norm{h}^{2}-\epsilon\int_{\reals}
        \norm{(L-(k-\I\epsilon)I)^{-1}h}^{2}dk.
    \end{align*}
Taking supremum with respect to $\varepsilon,$ it follows that
\begin{equation*}
  \norm{\Gamma_{0}(L-\cdot I)^{-1}h}^{2}_{H_{2}^{-}(\mathcal{E})}\le\pi\norm{h}^{2}.
\end{equation*}
The second inequality in (\ref{trian}) of the lemma is proven in the same way.
\end{proof}

As mentioned in Section~\ref{sec:triples-bvp}, the characteristic
function $S$, given in \eqref{eq:def-characteristic-of-L}, has
nontangential limits almost everywhere on the real line in the strong
topology. Thus, for a two-component vector function
$\binom{\tilde{g}}{g}\in L^2({\mathbb R}, \mathcal{E})\oplus L^2({\mathbb R}, \mathcal{E}),$ the integral\footnote{This is in fact the same construction as proposed by \cite{Drogobych} and further developed by \cite{MR573902}. Henceforth in this section we follow closely the analysis of the named two papers, facilitated by the fact that essentially this way to construct the functional model only relies upon the characteristic function $S$ of the maximal dissipative operator and an estimate of the type claimed in Lemma \ref{lem:bounded-hardy} above. A similar argument for extensions of symmetric operators, based on the theory of boundary triples, was developed in \cite{MR2330831}, \cite{ChKS_OTAA}.}

\begin{equation}
  \label{eq:inner-in-functional}
  \int_\reals\inner{\begin{pmatrix} I & S^*(s)\\
      S(s) &
      I\end{pmatrix}
    \binom{\widetilde{g}(s)}{g(s)}}{\binom{\widetilde{g}(s)}{g(s)}}_{\mathcal{E}\oplus
    \mathcal{E}}ds,
\end{equation}
makes sense and is nonnegative due to the contractive properties of $S$.
The space
\begin{equation*}
\mathfrak{H}:=L^2\Biggl(\mathcal{E}\oplus\mathcal{E}; \begin{pmatrix}
    I & S^*\\
    S & I
  \end{pmatrix}\Biggr)
\end{equation*}
is the completion of the linear set of two-component vector functions
$\binom{\tilde{g}}{g}: {\mathbb R}\to \mathcal{E}\oplus \mathcal{E}$ with respect to the norm (\ref{eq:inner-in-functional}), where a factorisation by vectors on which (\ref{eq:inner-in-functional}) vanishes is assumed.
Naturally, not every element of the set can be identified with a pair
$\binom{\tilde{g}}{g}$ of two independent
functions, however we keep the notation
$\binom{\tilde{g}}{g}$ for the elements of this space.

  Another consequence of the contractive properties of the
  characteristic function $S$ is the inequalities
\begin{equation*}
  \norm{\widetilde{g}+S^*g}_{L^2(\reals,\mathcal{E})}\le\norm{\binom{\widetilde{g}}{g}}_\mathfrak{H},\quad
  \norm{S\widetilde{g}+g}_{L^2(\reals,\mathcal{E})}\le\norm{\binom{\widetilde{g}}{g}}_\mathfrak{H}\qquad\forall \widetilde{g}, g\in L^2({\mathbb R}, \mathcal{E}).
  \end{equation*}
  They imply, in particular, that for every sequence
  $\{\binom{\tilde{g}_n}{g_n}\}_{n=1}^\infty$ that is Cauchy with
  respect to the $\mathfrak{H}$-topology and
  such that $\widetilde{g}_n,g_n\in L^2(\reals,\mathcal{E})$ for all
  $n\in\nats$, the limits of $\widetilde{g}_n+S^*g_n$ and
  $S\widetilde{g}_n+g_n$ exists in $L^2(\reals,\mathcal{E})$, so that the
  objects $\widetilde{g}+S^*g$ and
  $S\widetilde{g}+g$ can always be treated as $L^2(\reals,\mathcal{E})$ functions.\footnote{In general, $\widetilde{g}+S^*g$ and $S\widetilde{g}+g$ are not independent of each other, see \cite{KisNab}.}

Consider the following subspaces of $\mathfrak{H}:$\footnote{In the language of scattering theory \cite{MR0217440}, the subspaces 
$\mathfrak{D}_-,$ $\mathfrak{D}_+$ are ``incoming" and ``outgoing" subspaces, respectively, for the group of translations of $\mathfrak{H},$ as was first observed in \cite{MR0365199}.}
\begin{equation}
  \label{eq:D-spaces}
  \mathfrak{D}_-:=
  \begin{pmatrix}
    0\\
    {H}^2_-(\mathcal{E})
  \end{pmatrix},\quad
 \mathfrak{D}_+:=
  \begin{pmatrix}
   {H}^2_+(\mathcal{E})\\
   0
 \end{pmatrix}.
\end{equation}
It is easily seen \cite{Drogobych} that the spaces $\mathfrak{D}_-$
and $\mathfrak{D}_+$ are mutually orthogonal in $\mathfrak{H}$.

Define the subspace
\begin{equation}
  \label{eq:definition-of-K}
  K:=\mathfrak{H}\ominus(\mathfrak{D}_-\oplus \mathfrak{D}_+),
\end{equation}
which is characterised as follows (see \cite{MR0365199, Drogobych}):
\begin{equation*}
  K=\left\{\begin{pmatrix}
   \widetilde{g}\\
   g
  \end{pmatrix}\in\mathfrak{H}: \widetilde{g}+S^*g\in  H^2_-({\mathcal E}),\ S\widetilde{g}+g\in
 H^2_+({\mathcal E})\right\}\,.
\end{equation*}
The orthogonal projection $P_K$ onto 
$K$ is given by (see \eg \cite{MR0500225})
\begin{equation}
\label{eq:pk-action}
  P_K
  \begin{pmatrix}
    \widetilde{g}\\
    g
  \end{pmatrix}
=
\begin{pmatrix}
  \widetilde{g}-P_+(\widetilde{g}+S^*g)\\[0.5em]
  g-P_-(S\,\widetilde{g}+g)
\end{pmatrix},
\end{equation}
where $P_\pm$ are the orthogonal Riesz projections in $L^2(E)$ onto
$H^2_\pm(E)$.

\begin{definition}[\cite{MR2330831}]
  \label{def:F-mappings}
 The mappings $\mathscr{F}_\pm:\mathscr{H}\to
  L_2(\reals,\mathcal{E})$ are defined by
  \begin{equation*}
    \mathscr{F}_+\begin{pmatrix}
    v_-\\ v\\ v_+
  \end{pmatrix}
    :=-\frac{1}{\sqrt{\pi}}\Gamma_0(L-(\cdot-\I
    0)I)^{-1}v+\bigl(S^*\widehat{v}_-+\widehat{v}_+\bigr)(\cdot)
  \end{equation*}
  and
  \begin{equation*}
    \mathscr{F}_-\begin{pmatrix}
    v_-\\ v\\ v_+
  \end{pmatrix}
    :=-\frac{1}{\sqrt{\pi}}\Gamma_0(L^*-(\cdot+\I
    0)I)^{-1}v+\bigl(\widehat{v}_-+S\widehat{v}_+\bigr)(\cdot).
  \end{equation*}
\end{definition}

Based on the above definition, we will now introduce a map from ${\mathscr H}$ to ${\mathfrak H},$ which will prove to be unitary. We will then show that ${\mathfrak H}$ serves as a representation space for the spectral form of the functional model discussed in Section \ref{sec:dilation-boundary-problem}. We implement this strategy in Lemmata \ref{lem:f-pm-phi}--\ref{lem:phi-on-to}.
\begin{lemma}
  \label{lem:f-pm-phi}
  Fix $z_{+}\in\complex_{+},$ $z_{-}\in\complex_{-},$ and consider the map
  \[
 \Phi: \begin{pmatrix}L_2(\reals_-,\mathcal{E})\\[0.4em]
  \mathcal{G}(z_+,z_-)\\[0.4em]
  L_2(\reals_+,\mathcal{E})\end{pmatrix}\ni
    \begin{pmatrix}
      v_-\\ v\\ v_+
    \end{pmatrix}\mapsto
    \begin{pmatrix}
      \widehat{v}_+(\cdot)+\dfrac{\I}{\sqrt{2\pi}}
      \left[\dfrac{1}{\cdot-z_-}S^*(\cc{z}_-)w_--\dfrac{1}{\cdot-z_+}w_+\right]\\[1.2em]
       \widehat{v}_-(\cdot)-\dfrac{\I}{\sqrt{2\pi}}\left[
        \dfrac{1}{\cdot-z_-}w_--\dfrac{1}{\cdot-z_+}S(z_+)w_+\right]
    \end{pmatrix}\in\mathfrak{H},
  \end{equation*}
  where $w_+,$ $w_-\in\mathcal{E}$ are determined uniquely, by Remark 2, from
  \begin{equation}
 \label{eq:v-in-G}
    v=\sqrt{2}\gamma(z_+)(M(z_+)+\I
    I)^{-1}w_++\sqrt{2}\gamma(z_-)(M(z_-)-\I I)^{-1}w_-.
  \end{equation}

The map $\Phi$ satisfies
\begin{equation}
  \label{eq:phi-f-pm}
  \begin{pmatrix}
    I & S^* \\
    S & I
  \end{pmatrix}
  \Phi\begin{pmatrix}
      v_-\\ v\\ v_+
    \end{pmatrix}=
    \begin{pmatrix}
      \mathscr{F}_+
      \begin{pmatrix}
      v_-\\[-1mm] v\\[-1mm] v_+
    \end{pmatrix}\\[5mm]
    \mathscr{F}_-
      \begin{pmatrix}
      v_-\\[-1mm] v\\[-1mm] v_+
    \end{pmatrix}
    \end{pmatrix},\qquad\begin{pmatrix}
      v_-\\ v\\ v_+
    \end{pmatrix}\in\begin{pmatrix}L_2(\reals_-,\mathcal{E})\\[0.4em]
  \mathcal{G}(z_+,z_-)\\[0.4em]
  L_2(\reals_+,\mathcal{E})\end{pmatrix}.
\end{equation}
\end{lemma}
\begin{proof}

  Taking into account Definition~\ref{def:F-mappings}, one immediately verifies that
  \eqref{eq:phi-f-pm} holds for $v=0$. Since $\Phi$, $\mathscr{F}_\pm$ are linear, it only remains to
  prove the assertion when $v_-=0\in L_2(\reals_-,\mathcal{E}),$ $v_+=0\in L_2(\reals_+,\mathcal{E}).$ Under this assumption, consider
  the first row in the  vector equality \eqref{eq:phi-f-pm}, where $v$ is replaced by the formula (\ref{eq:v-in-G}):
   \begin{equation}
   \begin{aligned}
       &\dfrac{\I}{\sqrt{2\pi}}\left[
      \dfrac{1}{\cdot-z_-}S^*(\cc{z}_-)w_--
      \dfrac{1}{\cdot-z_+}w_+\right] -\dfrac{\I}{\sqrt{2\pi}}S^*(\cdot)
      \left[\dfrac{1}{\cdot-z_-}w_--\dfrac{1}{\cdot-z_+}S(z_+)w_+\right]\\[0.5em]
    &=-\sqrt{\dfrac{2}{\pi}}\Gamma_0\bigl(L-(\cdot-\I0)I\bigr)^{-1}\left[\gamma(z_+)\bigl(M(z_+)+\I
    I\bigr)^{-1}w_++\gamma(z_-)\bigl(M(z_-)-\I I\bigr)^{-1}w_-\right]
  \end{aligned}
  \label{need_to_show}
  \end{equation}
In what follows, we show that
\begin{equation}
  \label{eq:to-satisfy-for-w-}
  \dfrac{\I}{\cdot-z_-}\bigl[S^*(\cdot)-S^*(\cc{z}_-)\bigr]w_-=
  2\Gamma_0\bigl(L-(\cdot-\I0)I\bigr)^{-1}\gamma(z_-)\bigl(M(z_-)-\I I\bigr)^{-1}w_-
\end{equation}
and
\begin{equation}
  \label{eq:to-satisfy-for-w+}
  \dfrac{\I}{\cdot-z_+}\bigl[I-S^*(\cdot)S(z_+)\bigr]w_+=
  2\Gamma_0\bigl(L-(\cdot-\I0)I\bigr)^{-1}\gamma(z_+)\bigl(M(z_+)+\I I\bigr)^{-1}w_+,
\end{equation}
and therefore (\ref{need_to_show}) holds, as required. To verify (\ref{eq:to-satisfy-for-w-}) first, consider $z\in\complex_{-}$.  Using the second resolvent identity, it follows
from \eqref{eq:adjoint-characteristic-of-L} that
\begin{equation}
\begin{aligned}
  \dfrac{1}{z-z_-}\bigl[S^*(\cc{z})-S^*(\cc{z}_-)\bigr]&=
  \dfrac{2\I}{z-z_{-}}\Bigl[\bigl(M(z)-\I I\bigr)^{-1}-\bigl(M(z_-)-\I I\bigr)^{-1}\Bigr]\\
  &= \dfrac{2\I}{z-z_{-}}\bigl(M(z)-\I I\bigr)^{-1}\bigl[M(z)-M(z_-)\bigr]\bigl(M(z_-)-\I I\bigr)^{-1}
  \end{aligned}
  \label{eq:auxiliar-difference-m}
\end{equation}
Therefore, by \eqref{eq:difference-m-functions}, \eqref{eq:Gamma-0-L-through-M-gamma},
one has
\begin{equation}
\label{eq:already-done}
\begin{split}
  \frac{1}{z-z_-}\bigl[S^*(\cc{z})-S^*(\cc{z}_-)\bigr]&=
  2\I(M(z)-\I I)^{-1}\gamma(\cc{z})^*\gamma(z_-)\bigl(M(z_-)-\I I\bigr)^{-1}\\
   &= -2\I\Gamma_0(L-z I)^{-1}\gamma(z_-)\bigl(M(z_-)-\I I\bigr)^{-1}\,.
   \end{split}
 \end{equation}
Passing to the limit as $z$ approaches a real value, we infer that 
\eqref{eq:to-satisfy-for-w-} is satisfied for all $w_-\in\mathcal{E}$. To prove \eqref{eq:to-satisfy-for-w+} for all $w_+\in\mathcal{E}$, we proceed in a similar way. By
straightforward calculations, one has, for $z\in{\mathbb C}_-,$
\begin{align*}
  \frac{1}{z-z_+}\bigl[I -S^*(\cc{z})S(z_+)\bigr]
&=-\frac{2\I}{z-z_{+}}
  \Bigl[\bigl(M(z)-\I I\bigr)^{-1}-\bigl(M(z_+)+\I I\bigr)^{-1}
  - 2\I\bigl(M(z)-\I I\bigr)^{-1}\bigl(M(z_+)+\I I\bigr)^{-1}\Bigr]\\
  &=-\frac{2\I}{z-z_{+}}\bigl(M(z)-\I I\bigr)^{-1}\bigl[(M(z_+)+\I
    I)-\bigl(M(z)-\I I\bigr)-2\I I\bigr]\bigl(M(z_+)+\I I\bigr)^{-1}\\
  & =\frac{2\I}{z-z_{+}}\bigl(M(z)-\I I\bigr)^{-1}
    \bigl[M(z)-M(z_+)\bigr]\bigl(M(z_+)+\I I\bigr)^{-1}
\end{align*}
Proceeding in the same way as (\ref{eq:already-done}) was obtained from (\ref{eq:auxiliar-difference-m}), one obtains
\begin{equation*}
  \frac{1}{z-z_+}\bigl(I -S^*(\cc{z})S(z_+)\bigr)
  =-2\I\Gamma_0(L-z I)^{-1}\gamma(z_+)\bigl(M(z_+)-\I I\bigr)^{-1},
\end{equation*}
which, by passing to the limit as $z$ approaches the real line, yields the required property. 


The second entry of the vector equality \eqref{eq:phi-f-pm} is proved
in a similar way.
\end{proof}
\begin{lemma}
  \label{lem:isometry-Ds}
The mapping $\Phi$, given in Lemma~\ref{lem:f-pm-phi},
is an isometry from
$L_2(\reals_-,\mathcal{E})\oplus\{0\}\oplus L_2(\reals_+,\mathcal{E})$ onto $\mathfrak{D}_-\oplus\mathfrak{D}_+$.
\end{lemma}
\begin{proof}
  Clearly, for all $v_\pm\in L_2(\reals_\pm,\mathcal{E}),$ 
  one has
  \begin{equation*}
    \norm{\Phi  \begin{pmatrix}
      v_-\\0\\0
    \end{pmatrix}}_{\mathfrak{H}}=
    \norm{\begin{pmatrix}
      0\\\widehat{v}_-
    \end{pmatrix}}_{\mathfrak{H}},\qquad
    \norm{\Phi  \begin{pmatrix}
      0\\0\\v_+
    \end{pmatrix}}_{\mathfrak{H}}=
     \norm{\begin{pmatrix}
     \widehat{v}_+\\0
    \end{pmatrix}}_{\mathfrak{H}}.
\end{equation*}
Thus, taking into account that the spaces $\mathfrak{D}_-$ and
$\mathfrak{D}_+$ are orthogonal (see the discussion following the formula (\ref{eq:D-spaces})), one has
\begin{equation*}
  \norm{\Phi  \begin{pmatrix}
      v_-\\0\\v_+
    \end{pmatrix}}^2_{\mathfrak{H}}=
\norm{\begin{pmatrix}
      0\\\widehat{v}_-
    \end{pmatrix}}^2_{\mathfrak{H}}
+\norm{\begin{pmatrix}
     \widehat{v}_+\\0
    \end{pmatrix}}^2_{\mathfrak{H}}\,.
\end{equation*}
Finally note that
\begin{equation*}
  \norm{\begin{pmatrix}
      0\\\widehat{v}_-
    \end{pmatrix}}_{\mathfrak{H}}=\norm{\widehat{v}_-}_{L_2(\reals,\mathcal{E})}
  =\norm{v_-}_{L_2(\reals_-,\mathcal{E})},\qquad
\norm{\begin{pmatrix}
      \widehat{v}_+\\ 0
    \end{pmatrix}}_{\mathfrak{H}}=\norm{\widehat{v}_+}_{L_2(\reals,\mathcal{E})}
  =\norm{v_+}_{L_2(\reals_+,\mathcal{E})}\,.
\end{equation*}
The surjectivity of the mapping follows from the fact that the Fourier
transform is a unitary mapping between $L_2(\reals_\pm,\mathcal{E})$ and
$H^2_\pm(\mathcal{E}),$ by the Paley-Wiener theorem.
\end{proof}
\begin{lemma}
  \label{lem:isometry-phi}
The mapping $\Phi$, given in Lemma~\ref{lem:f-pm-phi} and extended
by linearity to
\begin{equation}
  \label{eq:dense-space-in-functional-model}
   L_2(\reals_-,\mathcal{E})\oplus   \Span_{z_\pm\in{\mathbb C}_\pm}\mathcal{G}(z_+,z_-)\oplus L_2(\reals_+,\mathcal{E})
\end{equation}
is an isometry from the set (\ref{eq:dense-space-in-functional-model}) to
$\mathfrak{H}$.
\end{lemma}
\begin{proof}
  Due to \eqref{eq:definition-of-K} and Lemma~\ref{lem:isometry-Ds},
  the assertion will be proved if one shows first that
\begin{equation}\label{eq:phi-leaves-K}
  \Phi
  \begin{pmatrix}
    0\\[0.3em] \Span_{z_\pm\in{\mathbb C}_\pm}\mathcal{G}(z_+,z_-)\\[0.8em]
  0
  \end{pmatrix}
\subset K
\end{equation}
and, second, that for all $z_\pm\in\complex_\pm$ and $v\in\mathcal{G}(z_+,z_-)$ 
one has
\begin{equation}
\label{eq:isometry-phi-in-K}
  \norm{\Phi
    \begin{pmatrix}
      0\\v\\0
    \end{pmatrix}
}_{\mathfrak{H}}=\norm{v}_{\mathcal{H}}\,.
\end{equation}
In view of the definition of $\Phi,$ see Lemma \ref{lem:f-pm-phi}, to establish \eqref{eq:phi-leaves-K} it suffices to verify that, for $z_\pm\in{\mathbb C}_\pm$
and $w_+, w_-\in{\mathcal E}$ chosen as in \eqref{eq:v-in-G}, the vectors
\begin{equation*}
  \frac{1}{\cdot-z_-}
  \begin{pmatrix}
    S^*(\cc{z}_-)w_-\\[0.3em]-w_-
  \end{pmatrix},\qquad
\frac{1}{\cdot-z_+}
\begin{pmatrix}
  -w_+\\[0.3em] S(z_+)w_+
\end{pmatrix}
\end{equation*}
are orthogonal to $\mathfrak{D}_-\oplus\mathfrak{D}_+$. To this end, consider $h_\pm\in H_2^\pm(\mathcal{E})$. Taking into account the fact that
\begin{equation}
 \label{eq:fraction-hardy+}
 (\cdot-z_-)^{-1}w_-\in H_2^+(\mathcal{E}),
\end{equation}
we obtain
\begin{align}
\inner{\frac{1}{\cdot-z_-}
  \begin{pmatrix}
    S^*(\cc{z}_-)w_-\\ -w_-
  \end{pmatrix}}
{
  \begin{pmatrix}
    h_+\\ h_-
  \end{pmatrix}
}_\mathfrak{H}&=\inner{\frac{1}{\cdot-z_-}S^*(\cc{z}_-)w_-}{h_++S^*h_-}_{L_2(\mathcal{E})}
    -\inner{\frac{1}{\cdot-z_-}w_-}{Sh_++h_-}_{L_2(\mathcal{E})}\nonumber\\[0.6em]
&=\inner{\frac{1}{\cdot-z_-}S^*(\cc{z}_-)w_-}{h_+}_{L_2(\mathcal{E})}
-\inner{\frac{1}{\cdot-z_-}w_-}{Sh_+}_{L_2(\mathcal{E})}\nonumber
\\[0.6em]
&
=-\inner{\frac{S^*(\cdot)-S^*(\cc{z}_-)}{\cdot-z_-}w_-}{h_+}_{\!\! L_2(\mathcal{E})}.\label{vanishexpr}
\end{align}
Now analytically continuing the function $S^*$ to the lower half-plane and using the fact that 
\[
\frac{S^*(\cdot)-S^*(\cc{z}_-)}{\cdot-z_-}w_-\in H^-_2(\mathcal{E}),
\]
we conclude that the expression (\ref{vanishexpr}) vanishes, as required.

In the same way, since
\begin{equation}
 \label{eq:fraction-hard-}
 (\cdot-z_+)^{-1}w_+\in H_2^-(\mathcal{E}),
\end{equation}
we conclude that
\begin{equation*}
  \inner{\frac{1}{\cdot-z_+}
  \begin{pmatrix}
    -w_+\\ S(z_+)w_+
  \end{pmatrix}}
{
  \begin{pmatrix}
    h_+\\ h_-
  \end{pmatrix}
}_\mathfrak{H}=-\inner{\frac{S(\cdot)-S(z_+)}{\cdot-z_+}w_+}{h_-}_{L_2(\mathcal{E})}=0.
\end{equation*}

It remains to prove (\ref{eq:isometry-phi-in-K}). In view of
Lemma~\ref{lem:f-pm-phi} and Definition~\ref{def:F-mappings}, one has
\begin{align*}
    \norm{\Phi
    \begin{pmatrix}
      0\\v\\0
    \end{pmatrix}
}_{\mathfrak{H}}^2&=\inner{\begin{pmatrix}
    I & S^* \\
    S & I
  \end{pmatrix}
  \Phi\begin{pmatrix}
      0\\ v\\ 0
    \end{pmatrix}}{\Phi\begin{pmatrix}
      0\\v\\0
    \end{pmatrix}}_{
    L_2(\mathcal{E})\oplus L_2(\mathcal{E})}
=\left\langle\begin{pmatrix}
      \mathscr{F}_+
      \begin{pmatrix}
      0\\
      v\\
      0
    \end{pmatrix}\\[5mm]
    \mathscr{F}_-
      \begin{pmatrix}
      0\\
      v\\
      0
    \end{pmatrix}
    \end{pmatrix}, \Phi\begin{pmatrix}
      0\\v\\0
    \end{pmatrix}\right\rangle
    \\[1mm]
&
=
\inner{\begin{pmatrix}
      -\dfrac{1}{\sqrt{\pi}}\Gamma_0(L-(\cdot-\I0)I)^{-1}v\\[5mm]
    -\dfrac{1}{\sqrt{\pi}}\Gamma_0(L^*-(\cdot+\I0)I)^{-1}v
\end{pmatrix}
}{\begin{pmatrix}
      \dfrac{\I}{\sqrt{2\pi}}\left[
        \dfrac{1}{\cdot-z_-}S^*(\cc{z}_-)w_--\dfrac{1}{\cdot-z_+}w_+\right]\\[5mm]
       -\dfrac{\I}{\sqrt{2\pi}}\left[
        \dfrac{1}{\cdot-z_-}w_--\dfrac{1}{\cdot-z_+}S(z_+)w_+\right]
    \end{pmatrix}}.
\end{align*}
By Lemma~\ref{lem:bounded-hardy}, one has $\Gamma_0(L-\cdot)^{-1}v\in H_2^-(\mathcal{E})$ and $\Gamma_0(L^*-\cdot)^{-1}v\in H_2^+(\mathcal{E})$. Thus, in view of (\ref{eq:fraction-hardy+})
  and (\ref{eq:fraction-hard-}), one obtains, using the Cauchy formula for Hardy classes, that
  \begin{align}
   \norm{\Phi
    \begin{pmatrix}
      0\\v\\0
    \end{pmatrix}
}_{\mathfrak{H}}^2&=
\left\langle\begin{pmatrix}
      -\dfrac{1}{\sqrt{\pi}}\Gamma_0(L-(\cdot-\I0)I)^{-1}v\\[5mm]
    -\dfrac{1}{\sqrt{\pi}}\Gamma_0(L^*-(\cdot+\I0)I)^{-1}v
\end{pmatrix},
\begin{pmatrix}
      -\dfrac{\I}{\sqrt{2\pi}}\dfrac{1}{\cdot-z_+}w_+\\[5mm]
       -\dfrac{\I}{\sqrt{2\pi}}
        \dfrac{1}{\cdot-z_-}w_-
    \end{pmatrix}\right\rangle_{\hspace{-1mm}L_2(\mathcal{E})\oplus
    L_2(\mathcal{E})}
   \nonumber\\[0.4em]
    &=-\sqrt{2}\bigl(\bigl\langle\Gamma_{0}(L-\cc{z}_{+}I)v, w_{+}\bigr\rangle+\bigl\langle\Gamma_{0}(L^{*}-\cc{z}_{-}I)v, w_{-}\bigr\rangle
      \bigr)
      \nonumber\\[0.6em]
   &
    =\sqrt{2}\bigl\langle v, \gamma(z_+)(M(z_+)+\I I)w_{+}+\gamma(z_-)(M(z_-)-\I
      I)w_{-}\bigr\rangle,\label{last_inner}
  \end{align}
  where for obtaining the last equality we use
  (\ref{eq:Gamma-0-L-through-M-gamma}). Due to (\ref{eq:v-in-G}), the formula (\ref{last_inner}) immediately implies
  \eqref{eq:isometry-phi-in-K},
  as required.
\end{proof}
Due to Lemma~\ref{lem:set-resolvent-dense-whole-space} and
Lemma~\ref{lem:isometry-phi}, the mapping $\Phi$ can be extended by continuity to
the whole space $\mathscr{H}$, provided that the operator $\widetilde{A}$ is simple. We will use same notation $\Phi$ for this extension.
\begin{lemma}
  For all $z\in{\mathbb C}\setminus{\mathbb R},$ one has
    $\Phi(\mathscr{A}-zI)^{-1}=(\cdot-z)^{-1}\Phi.$
\end{lemma}
\begin{proof}
  We prove the statement  for $z\in\complex_{+},$ as the case $z\in\complex_{-}$ is established in a similar way.  
  
  Consider an arbitrary
  $(h_-,h,h_+)^\top\in\mathscr{H},$ and let $(f_-,f,f_+)^\top$ be the vector
  defined by \eqref{eq:other-vector-in-domain}.
  It follows from (\ref{eq:h-pm-f-pm}) that
  \begin{equation}
    \label{eq:used-in-last-term}
    (\cdot-z)\widehat{f}_\pm(\cdot)=
    \widehat{h}_\pm(\cdot)\pm\frac{\I}{\sqrt{2\pi}}f_\pm(0)\,.
  \end{equation}
  Recall that $\widehat{h}_\pm$ and $\widehat{f}_\pm$ are the Fourier
  transforms of $h_\pm$ and $f_\pm$, respectively. According to
  Definition~\ref{def:F-mappings} and
  (\ref{eq:other-vector-in-domain}), one has
\begin{align*}
  \mathscr{F}_-\begin{pmatrix}
    f_-\\ f\\ f_+
  \end{pmatrix}&=
 -\dfrac{1}{\sqrt{\pi}}\Gamma_0(L^{*}-(\cdot+\I0)I)^{-1}f+ \widehat{f}_-(\cdot)+\bigl(S\widehat{f}_+\bigr)(\cdot) \nonumber\\[0.4em]
&=-\dfrac{1}{\sqrt{\pi}}\Gamma_0(L^{*}-(\cdot+\I0)I)^{-1}\left[(L^{*} -z I)^{-1}h
    +\sqrt{2}\gamma(z)(M(z)+\I I)^{-1}f_+(0)\right]\nonumber\\[0.5em]
&+\dfrac{1}{\cdot-z}\left[\bigl(\widehat{h}_-+S\widehat{h}_+\bigr)(\cdot)
 +\dfrac{\I}{\sqrt{2\pi}}\bigl(Sf_+(0)-f_-(0)\bigr)\right]
\end{align*}
where to obtain the expression in the second square brackets
we invoke \eqref{eq:used-in-last-term}. Thus, using the resolvent identity and (\ref{eq:def-characteristic-of-L}),
\begin{equation}
\begin{aligned}
  \mathscr{F}_-\begin{pmatrix}
    f_-\\ f\\ f_+
  \end{pmatrix}&=\frac{1}{\cdot-z}
\mathscr{F}_-\begin{pmatrix}
    h_-\\ h\\ h_+
  \end{pmatrix}
-\frac{1}{\sqrt{\pi}}\Gamma_0(L^{*}-(\cdot+\I
    0)I)^{-1}\gamma(z)\bigl(M(z)+\I I\bigr)^{-1}\sqrt{2}f_+(0)\nonumber\\[0.3em]
  & +
  \frac{1}{(\cdot-z)\sqrt{\pi}}
    \left(\Gamma_0(L^{*}-zI)^{-1}h +\frac{\I}{\sqrt{2}}\left[
    f_+(0)-f_-(0)-2\I\bigl(M(\cdot+\I 0\bigr)+\I I)^{-1}
    f_+(0)\right]\right).
\end{aligned}
      \label{eq:long-expression}
\end{equation}
Consider the third term on the right-hand side of (\ref{eq:long-expression}) evaluated at $\zeta\in\complex_{+}.$
Using the property ({\it cf.} (\ref{two_ast}))
\begin{equation*}
  f_{+}(0)-f_{-}(0)=\sqrt{2}\I\Gamma_{0}f,
\end{equation*}
we write it as follows:
\begin{align}
  \frac{1}{(\zeta-z)\sqrt{\pi}}&
    \left(\Gamma_0(L^{*}-zI)^{-1}h +\frac{\I}{\sqrt{2}}\left[
    f_+(0)-f_-(0)-2\I(M(\zeta)+\I I)^{-1}
    f_+(0)\right]\right)\nonumber\\[0.1em]
  &=\frac{1}{(\zeta-z)\sqrt{\pi}}
    \left(\Gamma_0(L^{*}-zI)^{-1}h-\Gamma_{0}f+
    \sqrt2(M(\zeta)+\I I)^{-1}f_+(0)\right)\nonumber\\[0.3em]
  &= \frac{\sqrt{2}}{(\zeta-z)\sqrt{\pi}}
    \left(-\Gamma_{0}\gamma(z)
    \bigl(M(z)+\I I\bigr)^{-1}f_+(0)+\bigl(M(\zeta)+\I I\bigr)^{-1}f_{+}(0)\right)\nonumber\\[0.4em]
  &=\frac{\sqrt{2}}{(\zeta-z)\sqrt{\pi}}\left(\bigl(M(\zeta)+\I I\bigr)^{-1}
    -\bigl(M(z)+\I I\bigr)^{-1}\right)f_{+}(0)\nonumber\\[0.4em]
  &=-\sqrt{\frac{2}{\pi}}\bigl(M(\zeta)+\I I\bigr)^{-1}
    \left(\frac{M(\zeta)-M(z)}{\zeta-z}\right)\bigl(M(z)+\I I\bigr)^{-1}f_{+}(0)
    \nonumber\\[0.4em]
  &
  =-\sqrt{\frac{2}{\pi}}\bigl(M(\zeta)+\I I\bigr)^{-1}
    \gamma(\cc{\zeta})^{*}\gamma(z)\bigl(M(z)+\I I\bigr)^{-1}f_{+}(0),\nonumber\\[0.4em]
  &=\sqrt{\frac{2}{\pi}}
  \Gamma_0(L^{*}-\zeta I)^{-1}\gamma(z)\bigl(M(z)+\I I\bigr)^{-1}f_+(0).\label{sixth}
\end{align}
where for the second equality $f$ is replaced by \eqref{eq:other-vector-in-domain}, while for the third and fourth equalities
we have used \eqref{eq:gamma-0-solution-operator} and the second resolvent identity, respectively. Furthermore, we utilise \eqref{eq:difference-m-functions} to obtain the fifth equality. The identities \eqref{eq:Gamma-0-L-through-M-gamma} now yield the final expression (\ref{sixth}).

It follows that the second and third terms on the right-hand side of \eqref{eq:long-expression} 
cancel each other as $\zeta$ approaches the real line. We have therefore shown that
\begin{equation}
  \mathscr{F}_{-}(\mathscr{A}-zI)^{-1}=(\cdot-z)^{-1}\mathscr{F}_{-},\quad z\in\complex_{+}.
  \label{F_minus}
\end{equation}
Similarly, one proves that
\begin{equation}
  \mathscr{F}_{+}(\mathscr{A}-zI)^{-1}=(\cdot-z)^{-1}\mathscr{F}_{+},\quad z\in\complex_{+}.
  \label{F_plus}
\end{equation}
Combining  (\ref{F_minus}), (\ref{F_plus}), and Lemma~\ref{lem:f-pm-phi} yields the claim.
\end{proof}
\begin{lemma}
  \label{lem:phi-on-to}
  The operator $\Phi$ maps $\mathscr{H}$ onto $\mathfrak{H}$ unitarily.
\end{lemma}
\begin{proof}
  In view of Lemma~\ref{lem:isometry-phi}, the mapping $\Phi$ is an isometry
  defined in the whole space $\mathscr{H}$. It thus suffices
  to show that the range of $\Phi$ is dense in $\mathfrak{H}$. To this end,
  suppose $\mathfrak{g}\in\mathfrak{H}$ is such that
  \begin{equation}
    \label{eq:orthogonal-complement-of-dense-set}
    \inner{\mathfrak{g}}{\Phi
      \begin{pmatrix}
        v_{-}\\ v\\ v_{+}
      \end{pmatrix}}_{\mathfrak{H}}=0\qquad\forall\begin{pmatrix}
        v_{-}\\ v\\ v_{+}
      \end{pmatrix}\in 
       \begin{pmatrix}
        L_2(\reals_-,\mathcal{E})\\[0.5em]
          \Span_{z_\pm\in{\mathbb C}_\pm}\mathcal{G}(z_+,z_-)\\[0.7em]
            L_2(\reals_+,\mathcal{E})
 \end{pmatrix}.
     \end{equation}
    By Lemma~\ref{lem:isometry-Ds} and the definition of the subspace $K,$ see \eqref{eq:definition-of-K},
    this is equivalent to the existence of a nonzero
    $\mathfrak{g}\in K$ such that
    \eqref{eq:orthogonal-complement-of-dense-set} holds with
    $v_{-}=0,$ $v_{+}=0.$ On the other hand,  since $\Phi^{*}\mathfrak{g}\in\mathcal{H},$ one has
    \begin{equation*}
      0=\inner{\mathfrak{g}}{\Phi
      \begin{pmatrix}
        0\\ v\\ 0
      \end{pmatrix}}_{\mathfrak{H}}
      =\inner{\Phi^{*}\mathfrak{g}}{v}_{\mathcal{H}}\qquad \forall v\in  \Span_{z_\pm\in{\mathbb C}_\pm}\mathcal{G}(z_+,z_-),
    \end{equation*}
  which by Lemma~\ref{lem:set-resolvent-dense-whole-space} yields $\Phi^{*}\mathfrak{g}=0,$ and hence $\mathfrak{g}=0.$
\end{proof}

Combining the above lemmata, we obtain the following result, concerning the representation of the dilation ${\mathscr A}$ as the operator of multiplication in the two-component model space $\mathfrak{H}.$

\begin{theorem}
Under the above definitions of ${\mathscr A}$ and $\Phi,$ one has
 \begin{align*}
    (L-zI)^{-1}=\Phi^*P_K(\cdot-z)^{-1}\bigr\vert_K\Phi,\qquad z\in{\mathbb C}_-,\\[0.4em]
   (\mathscr{A}-zI)^{-1}=\Phi^*(\cdot-z)^{-1}\Phi,\qquad z\in{\mathbb C}\setminus{\mathbb R},
  \end{align*}
  where $\Phi$ is unitary from ${\mathscr H}$ to ${\mathfrak H}.$

\end{theorem}

\section{Boundary traces of the resolvents of BVP}
\label{resolvent_section}

Our aim here is to derive an explicit formula for the solution
operator of the spectral boundary-value problem
\eqref{eq:general-spectral-bv-problem}.
To this end, consider the operator (see (\ref{Krein_formula1}), (\ref{Krein_formula}),
{\it cf.} \cite[Section 5]{Ryzh_spec})
\begin{equation*}
-\bigl(\overline{\alpha+\beta M(z)}\bigr)^{-1}\beta,
\end{equation*}
for all $z$ such that $0\in\rho(\alpha+\beta M(z))$. It is
convenient to assume that $\beta$ is boundedly invertible, which we do
henceforth. Recall, that above (see Section \ref{sec:triples-bvp}) we have also required that $\beta$ is bounded, and $\alpha$ is such that $\dom(\alpha)\supset \dom (\Lambda)$ and
$\alpha +\beta \Lambda $ is closable.

We note that $\widetilde{M}(z):=M(z)-\Lambda$ is bounded and 
\begin{equation}
\bigl(\overline{\alpha+\beta M(z)}\bigr)^{-1}\beta=\bigl(\overline{\alpha+\beta\Lambda}+\beta\widetilde{M}(z)\bigr)^{-1}\beta=\bigl(\beta^{-1}(\overline{\alpha+\beta\Lambda})+\widetilde{M}(z)\bigr)^{-1}.
\label{aux_closure}
\end{equation}
Furthermore, one has $\dom(\Lambda)\subset\dom(\alpha+\beta\Lambda),$ and
\[
\beta^{-1}\alpha+\Lambda\subset\beta^{-1}(\overline{\alpha+\beta\Lambda}).
\]
In addition, $\beta^{-1}(\overline{\alpha+\beta\Lambda})$ is closed, as a consequence of the general fact that whenever $T_1$ is bounded with a bounded inverse and $T_2$ is closed, the operator $T_1T_2$ is closed. Therefore, $\beta^{-1}\alpha+\Lambda$ is closable and
\begin{equation}
\overline{\beta^{-1}\alpha+\Lambda}=\beta^{-1}(\overline{\alpha+\beta\Lambda}).
\label{closure_property}
\end{equation}
Combining (\ref{aux_closure}) and (\ref{closure_property}), we obtain
\begin{equation}
\bigl(\overline{\alpha+\beta M(z)}\bigr)^{-1}\beta=\bigl(\overline{\beta^{-1}\alpha+\Lambda}+\widetilde{M}(z)\bigr)^{-1}=\bigl(\overline{B+M(z)}\bigr)^{-1},\quad B:=\beta^{-1}\alpha,
\label{B_star}
\end{equation}
and \cite[Theorem 5.1]{Ryzh_spec} implies that
\begin{equation}
\label{inclusion}
{\mathcal Q}_{B}:=\bigl\{z: 0\in\rho(\overline{B+M(z)})\bigr\}\subset\rho(A_{\alpha\beta}).
\end{equation}
For convenience, henceforth we use the notation
$Q_B(z):=-\bigl(\overline{B+M(z)}\bigr)^{-1},$ $z\in {\mathcal Q}_B.$

Notice that \cite[Theorem 5.1]{Ryzh_spec} requires ${\mathcal Q}_B\neq\emptyset,$
which cannot be guaranteed in the most general setup.
In the present article we focus on the PDE setting, where the standard choice of boundary conditions implies that $\Lambda$ is the Dirichlet-to-Neumann map \cite{Ryzh_spec}.
This allows us to make some reasonable assumptions that are bound to hold provided the boundary of the spatial domain in the BVP is smooth, so that  \cite[Theorem 5.1]{Ryzh_spec} is applicable and
the resulting operator $A_{\alpha\beta}$ has discrete spectrum in ${\mathbb C}_-\cup{\mathbb C}_+.$ In what follows, we utilise the standard notation $\mathfrak{S}_\infty$ the Banach algebra of compact operators \cite[Section 11]{MR1192782} on the boundary space $\mathcal{E}.$ 

\begin{lemma}
\label{B_condition}
Suppose that $\Lambda$ is the Dirichlet-to-Neumann map of a BVP problem, such that it is a self-adjoint operator with purely discrete spectrum, accumulating to $-\infty.$\footnote{Any BVP for a second-order elliptic PDE in a domain with smooth boundary has these properties, as follows from a straightforward argument based on the Poincar\'{e}-Wirtinger inequality and the Lax-Milgram lemma.}
 Then $M(z)^{-1}\in\mathfrak{S}_\infty$ for all  $z\in{\mathbb C}\setminus{\mathbb R}.$
\end{lemma}

\begin{proof}

Choose a finite-rank operator $K$ such that $\Lambda+K$ has trivial kernel and $(\Lambda+K)^{-1}\in\mathfrak{S}_\infty.$ Such a choice is obviously always possible. Furthermore, by the second Hilbert identity,
\[
M(z)^{-1}-(\Lambda+K)^{-1}=(\Lambda+K)^{-1}\Xi M(z)^{-1},
\]
where $\Xi$ is a bounded operator. Hence,
$M(z)^{-1}\in{\mathfrak S}_\infty.$
\end{proof}

\begin{corollary}
Within the conditions of Lemma \ref{B_condition}, if $B$ is bounded, then $BM(z)^{-1}\in\mathfrak{S}_\infty$ for all $z\in \mathbb{C}\setminus\mathbb{R}$.
\end{corollary}

\begin{remark}
  Note that if one drops the condition that $B$ is bounded, it is possible for $\mathcal Q_B$ to be empty. Indeed, put $\alpha=-\Lambda$ and $\beta=I$ (as shown in \cite{Ryzh_spec}, under these assumptions the operator $A_{\alpha\beta}$ is the 
  Kre\u\i n extension \cite{Ashbaugh_et_al} of the operator $\widetilde A$). Then by \eqref{M_operator} one has
\begin{equation*}
 B+ M(z)= z\Pi^*(I-zA_0^{-1})^{-1}\Pi,
\end{equation*}
which is shown to be compact under the assumptions of Lemma \ref{B_condition}. However, the following theorem suggests that instead of the 
restriction that $B$ be bounded, it suffices to assume that it is compact relative to $M(z),$ in order to ensure that $\mathcal Q_B$ coincides with $\mathbb{C}\setminus \mathbb{R}$ with the exception of a discrete set.
\end{remark}





\begin{theorem}
\label{conditional_lemma}
Suppose that  $BM(z)^{-1}\in\mathfrak{S}_\infty$ for at least one
$z\in{\mathbb C}_+$ and at least one $z\in{\mathbb C}_-$ (and hence at all $z\in{\mathbb C}\setminus{\mathbb R}$), where $B$ is defined by (\ref{B_star}). If
$I+BM(z)^{-1}$ is invertible for for at least one $z\in{\mathbb
    C}_+$ and at least one $z\in{\mathbb C}_-,$
then

1) The operator $A_{\alpha\beta}$ has at most discrete spectrum in
$\complex\setminus\reals$ (accumulating
at the real line only).

2) One has
$\rho(A_{\alpha\beta})={\mathcal Q}_{B}.$
\end{theorem}

\begin{proof} By the Analytic Fredholm Theorem, see \cite[Theorem 8.92]{RR}, the operator $I+BM(z)^{-1}$ is invertible at all $z\in\complex\setminus\reals$ with the exception of a discrete set of points. Therefore, for any $z$ such that the inverse exists, one has
\[
(\overline{B+M(z)})^{-1}=M(z)^{-1}\bigl(I+BM(z)^{-1}\bigr)^{-1}.
\]

This implies that the ``Kre\u\i n formula", {\it cf.} (\ref{Krein_formula1}), holds at all $z\in\complex\setminus\reals$ with the exception of a discrete set of points:
\begin{equation}
(A_{\alpha\beta}-z)^{-1}=(A_{0}-z I)^{-1}-(I-zA_{0}^{-1})^{-1}\Pi\bigl(\overline{B+M(z)}\bigr)^{-1}\Pi^{*},\qquad B=\beta^{-1}\alpha,
\label{Krein_formula}
\end{equation}
and therefore $\rho(A_{\alpha\beta})$ is discrete in $\complex\setminus\reals,$ which proves the first claim.

Furthermore, the right-hand side of (\ref{Krein_formula}) is analytic whenever its left-hand side is, {\it i.e.} on the set $\rho(A_{\alpha\beta}),$ which immediately implies the inclusion $\rho(A_{\alpha\beta})\subset{\mathcal Q}_{B}.$
The second claim of the theorem now follows by comparing this with (\ref{inclusion}).
\end{proof}




The formulae in the next lemma are analogous to
\cite[Eqs. (2.18), (2.22)]{MR2330831}.
\begin{lemma}
  \label{lem:ryzhov-formulae}
  Assume that $B$ defined by (\ref{B_star}) is bounded. Then the following identities hold:
  \begin{align}
    \Gamma_0(A_{\alpha\beta}-zI)^{-1}&=\Theta_{B}(z)^{-1}\Gamma_0(L-zI)^{-1}\quad\ \ \forall\,z\in\complex_-\cap{\mathcal Q}_{B},
    \label{eq:ii-kappa}
    \\[0.2em]
\Gamma_0(A_{\alpha\beta}-zI)^{-1}
&=\widehat{\Theta}_{B}(z)^{-1}\Gamma_0(L^*-zI)^{-1}\quad\forall\,z\in\complex_+\cap{\mathcal Q}_{B},
\label{eq:kappa-ii+}
\end{align}
where $\Theta_B$ and $\widehat{\Theta}_B$ are defined via their inverses:
\begin{align}
\Theta_B(z)^{-1}&=2{\rm i}Q_B(z)\bigl(I-S^*(\overline{z})\bigr)^{-1},\quad z\in\complex_-\cap{\mathcal Q}_B,
\label{theta_def}
\\[0.2em]
\widehat{\Theta}_B(z)^{-1}&=2{\rm i}Q_B(z)\bigl(I-S(z)\bigr)^{-1},\quad z\in\complex_+\cap{\mathcal Q}_B.
\label{theta_hat_def}
\end{align}
\end{lemma}
\begin{proof}
Fix an arbitrary $h\in\cH$ and define
\begin{equation*}
    g_{\alpha\beta}:=(A_{\alpha\beta}-zI)^{-1}h,\qquad z\in\rho(A_{\alpha\beta}),    
\end{equation*}
so that, in particular, $g_{-\I I\,I}=(L-zI)^{-1}h,$ $z\in{\mathbb C}_+,$ and $g_{\I I I}=(L^*-zI)^{-1}h,$ $z\in{\mathbb C}_-.$

In order to prove (\ref{eq:ii-kappa}),
  suppose that $z\in\complex_-\cap{\mathcal Q}_{B},$
  so the resolvents $(L-z
  I)^{-1}$ and $(A_{\alpha\beta}-z I)^{-1}$ are defined
  on the whole space $\cH$.
Clearly, the vector
  \begin{equation*}
  \label{g_definition}
    g:=g_{-\I I\,I}-g_{\alpha\beta}=\left((L-zI)^{-1}-(A_{\alpha\beta}-zI)^{-1}\right)h
  \end{equation*}
is an element of $\ker(A-zI).$
It follows from $g_{-\I I\,I}\in\dom(L)$ and $g_{\alpha\beta}\in\dom(A_{\alpha\beta})$  that
$
  \Gamma_1g_{-\I I\,I}=\I\Gamma_0g_{-\I I\,I}$ and $\beta\Gamma_1g_{\alpha\beta}=-\alpha\Gamma_0g_{\alpha\beta},$
and therefore one has
\begin{equation}
\begin{aligned}
  0=\beta\Gamma_1(g+g_{\alpha\beta})-\I\beta\Gamma_0(g+g_{\alpha\beta})
  &=\beta\Gamma_1g-\I\beta\Gamma_0g+\beta\Gamma_1g_{\alpha\beta}
  -\I\beta\Gamma_0g_{\alpha\beta}\\[0.4em]
  &=\beta M(z)\Gamma_0g-\I\beta\Gamma_0g-\alpha\Gamma_0g_{\alpha\beta}-\I\beta\Gamma_0g_{\alpha\beta},
\end{aligned}
 \label{latter_calculation}
 \end{equation}
where in the last equality we also use the fact that $g\in\ker(A-zI),$
together with Definition~\ref{def:m-function}.
Hence, by collecting the terms in the calculation (\ref{latter_calculation}), one has ({\it cf.} (\ref{g_definition}))
\begin{equation*}
\bigl(\alpha+\beta M(z)\bigr)\Gamma_0g=(\alpha+\I\beta)\Gamma_0(g+g_{\alpha\beta})=(\alpha+\I\beta)\Gamma_0g_{-\I I\, I},
\end{equation*}
which, in turn, implies that, for $z\in {\mathcal Q}_{B}$ one has
\begin{equation*}
  \Bigl\{I-\overline{\bigl({B}+M(z)\bigr)}^{-1}\bigl({B}+\I I\bigr)\Bigr\}\Gamma_0g_{-\I I\,I}=
\Gamma_0g_{\alpha\beta}.
\end{equation*}
Finally, using the second resolvent identity
\[
(\overline{B+\I I})^{-1}-\bigl(\overline{B+M(z)}\bigr)^{-1}=
\bigl(\overline{B+M(z)}\bigr)^{-1}\bigl(M(z)-\I I\bigr)(\overline{B+\I I})^{-1},
\]
we obtain
\begin{align*}
I-\overline{\bigl({B}+M(z)\bigr)}^{-1}\bigl({B}+\I I\bigr)&=\overline{\bigl({B}+M(z)\bigr)}^{-1}\bigl(M(z)-\I I\bigr)
=2{\rm i}Q_{B}(z)\bigl(I-S^*(\overline{z})\bigr)^{-1},
\end{align*}
where we use the formula (\ref{eq:adjoint-characteristic-of-L}).


The identity \eqref{eq:kappa-ii+} is proved by an argument similar to the above, where the vector $g_{-\I I\,I}$ is replaced with with $g_{\I I\,I},$ for $z\in{\mathbb C}_+,$ and the formula (\ref{eq:def-characteristic-of-L}) is used instead of (\ref{eq:adjoint-characteristic-of-L}).
\end{proof}

\begin{remark}
  Note that the boundedness condition imposed on $B$ in Lemma \ref{lem:ryzhov-formulae} can be relaxed. Not only can we assume that $B$ is such that $BM(z)^{-1}\in\mathfrak{S}_\infty,$ as suggested by Theorem \ref{conditional_lemma}, but the latter condition can be relaxed even further by assuming that $B$ is bounded relative to $M(z)$ with the bound\footnote{In the case when $B$ is compact relative to $M(z),$ the bound is zero, see \cite{Binding_Hryniv}.} less than 1 (see \cite{MR0407617}), which clearly suffices for $\overline{B+M(z)}=B+M(z)$. In present paper, however, we limit ourselves to physically motivated applications to BVP, which renders these considerations unnecessary. For this reason in what follows we will only consider the case when the parameter $B$ is bounded.

\end{remark}

\section{Functional model for non-necessarily dissipative operators}

In this section we obtain a useful representation for the resolvent of $A_{\alpha\beta}$ in the Hilbert space ${\mathfrak H},$ {\it i.e.} in the spectral functional model representation of $L.$ The results of this section generalise those of \cite{MR2330831}.
 We start by proving the following lemma. Throughout we assume that the condition imposed by Lemma \ref{B_condition} holds.

\label{model_sec}
\begin{lemma}
\label{lemma61}
Suppose that $B$ defined by (\ref{B_star}) is bounded, and denote
\begin{equation*}
\chi_B^\pm:=\frac{1}{2{\rm i}}(\pm B+{\rm i}I).
\end{equation*}
The following formulae hold for the functions defined in (\ref{theta_def}), (\ref{theta_hat_def}):
\begin{align}
\Theta_B(z)&=S^*(\overline{z})\chi_B^++\chi_B^-,\qquad z\in{\mathbb C}_-\cap{\mathcal Q}_B,\label{thetaprop}
\\[0.2em]
\widehat{\Theta}_B(z)&=S(z)\chi_B^-+\chi_B^+,\qquad z\in{\mathbb C}_+\cap{\mathcal Q}_B.\label{thetahatprop}
\end{align}
\end{lemma}
\begin{proof}
By the definition (\ref{theta_def}) and using the representation (\ref{eq:adjoint-characteristic-of-L}), we write, for $z\in\complex_-\cap {\mathcal Q}_B,$
\begingroup
\allowdisplaybreaks
\begin{align*}
\Theta_B(z)
&=-(2{\rm i})^{-1}\bigl(I-S^*(\overline{z})\bigr)\bigl(B+M(z)\bigr)=-(2{\rm i})^{-1}\bigl(I-S^*(\overline{z})\bigr)\bigl(B-2{\rm i}\bigl(I-S^*(\overline{z})\bigr)^{-1}+{\rm i}I\bigr)
\\[0.4em]
&
=-\bigl(I-S^*(\overline{z})\bigr)\chi_B^++I=S^*(\overline{z})\chi_B^++\chi_B^-,
\end{align*}
\endgroup
as claimed in (\ref{thetaprop}). Similarly, by the definition (\ref{theta_hat_def}) and using (\ref{eq:def-characteristic-of-L}), we obtain (\ref{thetahatprop}).
\end{proof}

The following is the main result of this section and is similar {\it in form} to \cite[Theorem 2.5]{MR2330831} and \cite[Theorem 3]{MR573902}. Its proof closely follows the lines of the mentioned works.
\begin{theorem}
  Suppose that $\beta^{-1}$ and $B=\beta^{-1}\alpha$ are bounded. Then
\begin{enumerate}[(i)]
\item  If $z\in\complex_-\cap{\mathcal Q}_{B}$
and
  $(\widetilde{g}, g)^\top\in K$, then
  \begin{equation*}
    \Phi(A_{\alpha\beta}-z I)^{-1}\Phi^{*}\left(\begin{matrix}\widetilde{g}\\[0.3em] g\end{matrix}\right)=P_K\frac{1}{\cdot-z}
\left(\begin{matrix}\widetilde{g}\\[0.4em] g-\chi_{B}^+\Theta_{B}(z)^{-1}(\widetilde{g}+S^*g)(z)\end{matrix}\right).
  \end{equation*}
\item  If $z\in\complex_+\cap{\mathcal Q}_{B}$
and
  $(\widetilde{g}, g)^\top\in K$, then
  \begin{equation*}
    \Phi(A_{\alpha\beta}-z I)^{-1}\Phi^{*}\left(\begin{matrix}\widetilde{g}\\[0.3em]g\end{matrix}\right)=P_K\frac{1}{\cdot-z}
\binom{\widetilde{g}-\chi_{B}^-\widehat{\Theta}_{B}(z)^{-1}(S\widetilde{g}+g)(z)}{g}\,.
  \end{equation*}
  Here, $(\widetilde{g}+S^*g)(z)$ and $(S\widetilde{g}+g)(z)$ denote
  the values at $z$ of the analytic continuations of the functions
  $\widetilde{g}+S^*g\in {H}^2_-({\mathcal E})$ and $S\widetilde{g}+g\in
  {H}^2_+({\mathcal E})$ into the lower half-plane and the upper half-plane, respectively.
\end{enumerate}
\end{theorem}

\begin{proof}

We prove (i).
The proof of (ii)
is carried out along the same lines.
For this
one should
establish the validity of the identities:
\begin{equation}
 \label{eq:validating-equalities}
  \mathscr{F}_\pm(A_{\alpha\beta}-z I)^{-1}\Phi^{-1}\left(\begin{matrix}\widetilde{g}\\[0.35em]
  g\end{matrix}\right)
=\mathscr{F}_\pm\Phi^{-1}P_K\frac{1}{\cdot-z}
\left(\begin{matrix}\widetilde{g}\\[0.4em]
g-\chi_B^+\Theta_B(z)^{-1}(\widetilde{g}+S^*g)(z)\end{matrix}\right),\qquad z\in\complex_-\cap{\mathcal Q}_{B}.
\end{equation}
First we compute the left-hand-side of
\eqref{eq:validating-equalities}.
It follows from Lemma~\ref{lem:ryzhov-formulae}
that for
$z,\lambda\in\complex_-\cap {\mathcal Q}_{B},$ $h\in\mathcal{H}$ one has
\begin{align*}
 \Gamma_0(L-zI)^{-1}(A_{\alpha\beta}-\lambda I)^{-1}h
&
=\Theta_{B}(z)\Gamma_0(A_{\alpha\beta} -z
I)^{-1}(A_{\alpha\beta}-\lambda I)^{-1}h\\[0.3em]
&=\frac{1}{z-\lambda}\Theta_{B}(z)\Gamma_0
\left[(A_{\alpha\beta}-z I)^{-1}-(A_{\alpha\beta}-\lambda I)^{-1}\right]h
\\[0.3em]
&
=\frac{1}{z-\lambda}\left[\Gamma_0(L-zI)^{-1}-
\Theta_{B}(z)\Gamma_0(A_{\alpha\beta}-\lambda
I)^{-1}\right]h\\[0.3em]
&=\frac{1}{z-\lambda}\left[\Gamma_0(L-zI)^{-1}-
\Theta_{B}(z)\Theta_{B}(\lambda)^{-1}\Gamma_0(L-\lambda I)^{-1}\right]h\,.
\end{align*}
Letting $z=k-\I\epsilon,$ $k\in\reals$, it follows from the above
calculation that
\begin{equation}
\begin{aligned}
  &\lim_{\epsilon\searrow 0}\Gamma_0(L-(k-\I\epsilon)I)^{-1}
(A_{\alpha\beta}-\lambda I)^{-1}h
\\[0.4em]
&=\lim_{\epsilon\searrow 0}
\frac{1}{(k-\I\epsilon)-\lambda}\left[\Gamma_0(L-(k-\I\epsilon)I)^{-1}-
\Theta_{B}(k-\I\epsilon)\Theta_{B}(\lambda)^{-1}\Gamma_0(L-\lambda I)^{-1}\right]h\,.
\end{aligned}
\label{lasteq}
\end{equation}
Combining the expression for $\mathscr{F}_+$ from Definition \ref{def:F-mappings} with (\ref{lasteq}) yields
\begin{equation*}
  \mathscr{F}_+(A_{\alpha\beta}-\lambda I)^{-1}h=
 \frac{1}{\cdot-\lambda}\left[(\mathscr{F}_+h)(\cdot)-
\Theta_{B}(\cdot)\Theta_{B}(\lambda)^{-1}\mathscr{F}_+h(\lambda)
\right]\,.
\end{equation*}
Hence, in view of the identity $\mathscr{F}_+ h=\widetilde{g}+S^*g,$ which follows from
\eqref{eq:v-in-G}, we obtain
\begin{equation}
\label{eq:f+on-resolvent}
  \mathscr{F}_+(A_{\alpha\beta}-\lambda I)^{-1}\Phi^{-1}\binom{\widetilde{g}}{g}=
 \frac{1}{\cdot-\lambda}\left[
 (\widetilde{g}+S^*g)(\cdot)-\Theta_{B}(\cdot)\Theta_{B}(\lambda)^{-1}(\widetilde{g}+S^*g)(\lambda)\right]\,.
\end{equation}

On the basis of Lemma~\ref{lem:ryzhov-formulae}
and reasoning in the same fashion as was done to write (\ref{eq:f+on-resolvent}), one verifies
\begin{equation}
  \label{eq:f-on-resolvent}
  \mathscr{F}_-(A_{\alpha\beta}-\lambda I)^{-1}\Phi^{-1}\left(\begin{matrix}\widetilde{g}\\[0.35em]g\end{matrix}\right)=
 \frac{1}{\cdot-\lambda}\bigl[
 (S\widetilde{g}+g)(\cdot)-\widehat{\Theta}_{B}(\cdot){\Theta}_{B}(\lambda)^{-1}(\widetilde{g}+S^*g)(\lambda)\bigr]\,.
\end{equation}

Let us focus on the right hand side of
\eqref{eq:validating-equalities}. Note that
\begin{align}
  &P_K\dfrac{1}{\cdot-z}
\left(\begin{array}{c}\widetilde{g}\\[0.5em]
g-\chi_{B}^+\Theta_{B}(z)^{-1}(\widetilde{g}+S^*g)(z)\end{array}\right)\nonumber\\[3mm]
 &=\left(\begin{array}{c}\dfrac{\widetilde{g}}{\cdot-z}-P_+\dfrac{1}{\cdot-z}
 \bigl[(\widetilde{g}+S^*g)(\cdot)-S^*\chi_{B}^+\Theta_{B}(z)^{-1}(\widetilde{g}+S^*g)(z)\bigr]
  \\[0.9em]\dfrac{1}{\cdot-z}\bigl(g-\chi_{B}^+\Theta_{B}(z)^{-1}(\widetilde{g}+S^*g)(z)\bigr)-P_-\dfrac{1}{\cdot-z}\bigl[(S\widetilde{g}+g)(\cdot)-\chi_{B}^+\Theta_{B}(z)^{-1}(\widetilde{g}+S^*g)(z)\bigr]\end{array}\right)\nonumber\\[4mm]
&=\dfrac{1}{\cdot-z}\left(\begin{array}{c}\widetilde{g}-(\widetilde{g}+S^*g)(z) +
S^*(\cc{z})\chi_{B}^+\Theta_{B}(z)^{-1}(\widetilde{g}+S^*g)(z)\\[0.9em]
g-\chi_{B}^+\Theta_{B}(z)^{-1}(\widetilde{g}+S^*g)(z)\end{array}\right),
\label{eq:rhs}
\end{align}
where (\ref{eq:pk-action}) is used in the first equality and in the
second the fact that if $f\in {H}_-^2(\mathcal{E})$, then, for all $z\in\complex_-$,
\begin{equation*}
  P_+\frac{f(\cdot)}{\cdot-z}=P_+\left(\frac{f(\cdot)+f(z)-f(z)}{\cdot-z}\right)=P_+\frac{f(z)}{\cdot-z}=\frac{f(z)}{\cdot-z}\,.
\end{equation*}

Now, apply $\mathscr{F}_+\Phi^{-1}$ to
(\ref{eq:rhs}) taking into account that $\mathscr{F}_+ h=\widetilde{g}+S^*g$ once again:
\begin{equation}
\begin{aligned}
\mathscr{F}_+\Phi^{-1}\frac{1}{\cdot-z}&\left(\begin{matrix}\widetilde{g}-(\widetilde{g}+S^*g)(z) +
S^*(\overline{z})\chi_{B}^+\Theta_{B}(z)^{-1}(\widetilde{g}+S^*g)(z)\\[0.7em]
g-\chi_{B}^+\Theta_{B}(z)^{-1}(\widetilde{g}+S^*g)(z)\end{matrix}\right)
\\[0.5em]
&=\frac{1}{\cdot-z}\bigl[(\widetilde{g}+S^*g)(\cdot)-(\widetilde{g}+S^*g)(z)+
(S^*(\overline{z})-S^*)\chi_{B}^+\Theta_{B}(z)^{-1}(\widetilde{g}+S^*g)(z)\bigr]
\\[0.5em]
&=\frac{1}{\cdot-z}\bigl[(\widetilde{g}+S^*g)(\cdot)-
\bigl(\Theta_{B}(z)-(S^*(\overline{z})-S^*)\chi_{B}^+\bigr)\Theta_{B}(z)^{-1}(\widetilde{g}+S^*g)(z)\bigr]\\[0.5em]
&=\frac{1}{\cdot-z}\bigl[(\widetilde{g}+S^*g)(\cdot)-\Theta_B(\cdot)\Theta_{B}(z)^{-1}(\widetilde{g}+S^*g)(z)\bigr],
\end{aligned}
\label{Fcomp}
\end{equation}
where for the last equality we have used Lemma \ref{lemma61}. By combining (\ref{Fcomp}) with
 (\ref{eq:f+on-resolvent}), we establish the first identity in (\ref{eq:validating-equalities}).

Finally, applying $\mathscr{F}_-\Phi^{-1}$ to
(\ref{eq:rhs}) and using the identity $\mathscr{F}_- h=S\widetilde{g}+g,$
we obtain
\begingroup
\allowdisplaybreaks
\begin{equation*}
\begin{aligned}
\mathscr{F}_-\Phi^{-1}&\frac{1}{\cdot-z}\left(\begin{array}{c}\widetilde{g}-(\widetilde{g}+S^*g)(z) +
S^*(\overline{z})\chi_{B}^+\Theta_{B}(z)^{-1}(\widetilde{g}+S^*g)(z)\\[0.8em]
g-\chi_{B}^+\Theta_{B}(z)^{-1}(\widetilde{g}+S^*g)(z)\end{array}\right)\\[0.6em]
&=\frac{1}{\cdot-z}\bigl[(S\widetilde{g}+g)(\cdot)-S(\widetilde{g}+S^*g)(z)-
\bigl(I-SS^*(\overline{z})\bigr)\chi_{B}^+\Theta_{B}(z)^{-1}(\widetilde{g}+S^*g)(z)\bigr]\\[0.3em]
&=\frac{1}{\cdot-z}\bigl[(S\widetilde{g}+g)(\cdot)-
\bigl(S\Theta_{B}(z)+\chi_{B}^+-SS^*(\overline{z})\chi_{B}^+\bigr)\Theta_{B}(z)^{-1}(\widetilde{g}+S^*g)(z)\bigr]\\[0.3em]
&=\frac{1}{\cdot-z}\bigl[(S\widetilde{g}+g)(\cdot)-
(S\chi_{B}^-+\chi_{B}^+)\Theta_{B}(z)^{-1}(\widetilde{g}+S^*g)(z)\bigr]
\\[0.3em]
&
=\frac{1}{\cdot-z}\bigl[(S\widetilde{g}+g)(\cdot)-\widehat{\Theta}_{B}(\cdot)\Theta_{B}(z)^{-1}(\widetilde{g}+S^*g)(z)\bigr],
\end{aligned}
\end{equation*}
\endgroup
where in the last two equalities we use Lemma \ref{lemma61}. Comparing this with
(\ref{eq:f-on-resolvent}), we arrive at the second identity in
(\ref{eq:validating-equalities}).
\end{proof}

\section{Application: a unitary equivalent model of an operator associated with BVP in a space with reproducing kernel}\label{sec:example}

In the present section we demonstrate that in the setting of operators of BVP, the results of Section 4 lead to the representation of $(L^*-zI)^{-1}$ as the Toeplitz operator $P_S\bigl(f(\cdot)(\cdot-z)^{-1}\bigr)\vert_{K_S},$ where $P_S$ is the orthogonal projection of $H^2_+({\mathcal E})$ onto $K_S:=H^2_+({\mathcal E})\ominus SH^2_+({\mathcal E}).$ Thus this results of Section 6 can be used to represent the resolvent of $A_{\alpha\beta}$ as a ``triangular'' perturbation of the aforementioned Toeplitz operator.

Throughout the section we assume that the condition imposed by Lemma \ref{B_condition} holds, the operator $B$ is bounded and that the operator $\widetilde{A}$ is simple.

The following proposition carries over together with its proof from \cite{KisNab}.

\begin{proposition}[\cite{KisNab}]
If the operator-function $S(z)$ is inner (which implies that its boundary values $S(k)$ are almost everywhere unitary on $\mathbb{R}$, see \cite{MR2760647}), then the Hilbert space $\mathcal{H}$ is unitary equivalent to the spaces $K_S:=H^2_+({\mathcal E})\ominus SH^2_+({\mathcal E})$ and
$K_S^\dagger:=H^2_-({\mathcal E})\ominus S^*H^2_-({\mathcal E}).$ Moreover, the unitary transformations of $\mathcal{H}$ to $K_S$ and $K_S^\dagger$ are given explicitly as the restrictions of the operators $\mathcal{F}_\mp$ of Definition \ref{def:F-mappings}, respectively, to the spaces $(0,\mathcal{H},0)$. For the spaces $K_S$ and $K_S^\dagger$ one additionally has the element-wise equality $S^*K_S=K_S^\dagger$.
\end{proposition}

\begin{remark}
It can be verified that the characteristic function $S$ is indeed inner if the spectrum of the operator $L$ is discrete. The latter is satisfied by the Kre\u\i n resolvent formula, provided that the conditions of Lemma \ref{B_condition} hold and the operator of the BVP with Dirichlet conditions has discrete spectrum, the latter being the case under minimal regularity conditions; however, see, {\it e.g.,} the discussion in \cite{Marletta} and references therein.
\end{remark}

The formula \eqref{eq:f+on-resolvent} applied to the operator $L$ and a similar computation in relation to the operator $L^*$ now yield the following result.
\begin{theorem}
The operator $(L-zI)^{-1}$ for $z\in\mathbb{C}_-$ is unitary equivalent to the Toeplitz operator $f\mapsto P_S^\dagger\bigl(f(\cdot)(\cdot -z)^{-1}\bigr)$ in the space $K_S^\dagger$; the operator $(L^*-zI)^{-1}$ for $z\in\mathbb{C}_+$ is unitary equivalent to the Toeplitz operator $f\mapsto P_S\bigl(f(\cdot)(\cdot -z)^{-1}\bigr)$ in the space $K_S$. Here $P_S^\dagger$ and $P_S$ are orthogonal projections onto $K_S^\dagger$ and $K_S$, respectively:
$$
P_S \frac{f(\cdot)}{\cdot-z}=P_+ \frac{f(\cdot)}{\cdot-z}=\frac{f(\cdot)-f(z)}{\cdot-z}, \quad z\in\mathbb{C_+},
$$
$$
P_S^\dagger \frac{f(\cdot)}{\cdot-z}=P_- \frac{f(\cdot)}{\cdot-z}=\frac{f(\cdot)-f(z)}{\cdot-z}, \quad z\in\mathbb{C_-},
$$
where $P_+,$ $P_-$ are orthogonal projections onto Hardy classes $H_+^2(\mathcal E),$ $H_-^2(\mathcal E)$, respectively.
\end{theorem}

For the operators of BVPs defined by different boundary conditions parameterised by the operator $B$, including self-adjoint ones, a similar argument yields the following representation.

\begin{theorem}
The operator $(A_{\alpha\beta}-z)^{-1}$ for $z\in\mathbb{C}_-\cap\rho(A_{\alpha\beta})$ is unitary equivalent to a ``triangular'' perturbation of the Toeplitz operator $f\mapsto P_S^\dagger\bigl(f(\cdot)(\cdot -z)^{-1}\bigr)$ in the space $K_S^\dagger$, namely, to the operator
$$
f\mapsto P_- \frac{f(\cdot)}{\cdot-z}  - P_-   \Theta_B(\cdot) \Theta_B(z)^{-1} P_+ \frac{f(\cdot)}{\cdot-z}.
$$
For $z\in\mathbb{C}_+\cap\rho(A_{\alpha\beta})$ the resolvent $(A_{\alpha\beta}-z)^{-1}$ is unitary equivalent to the operator
$$
f\mapsto P_+ \frac{f(\cdot)}{\cdot-z}  - P_+   \widehat{\Theta}_B(\cdot) \widehat{\Theta}_B(z)^{-1} P_- \frac{f(\cdot)}{\cdot-z}
$$
in the space $K_S$.
\end{theorem}

\begin{remark}
It is rather well-known that the spaces $K_S$  and $K_S^\dagger$ are Hilbert spaces with reproducing kernels, closely linked to the corresponding de Branges spaces in the ``scalar'' case of $\dim \mathcal E=1$. We refer the reader to the book \cite{Nikolski} for an in-depth survey of the subject area and of the related developments in modern complex analysis. The applications of the latter Theorem to the direct and inverse spectral problems of operators of BVPs is outside the scope of the present paper and will be dwelt upon elsewhere.
\end{remark}

\section*{Acknowledgements}


KDC is grateful for the financial support of
the Engineering and Physical Sciences Research Council: Grant
EP/L018802/2 ``Mathematical foundations of metamaterials:
homogenisation, dissipation and operator theory''.
AVK has been partially supported by the Russian Federation Government megagrant 14.Y26.31.0013
and by the  RFBR grant 19-01-00657-a.
LOS has been partially supported
by UNAM-DGAPA-PAPIIT IN110818 and SEP-CONACYT CB-2015 254062. LOS is grateful
for the financial support of PASPA-DGAPA-UNAM during his sabbatical
leave and thanks the University of Bath for their hospitality.


\begin{thebibliography}{9}


\bibitem{Agranovich} M. S. Agranovich, {\it Sobolev Spaces, Their Generalizations, and Elliptic Problems in Smooth and Lipschitz Domains,} Springer, 2015.


\bibitem{Amrein-Pearson} W. O. Amrein, D. B. Pearson, $M$-operators:
  a generalisation of Weyl-Titchmarsh
  theory. \emph{J. Comput. Appl. Math.,} 171(1-2):1--26, 2004.
  
\bibitem{Ashbaugh_et_al} M. S. Ashbaugh, F. Gesztesy, M. Mitrea, R. Shterenberg, G. Teschl, A survey on the Krein-von Neumann extension, the corresponding abstract buckling problem, and Weyl-type spectral asymptotics for perturbed Krein Laplacians in nonsmooth domains. Mathematical Physics, Spectral Theory and Stochastic Analysis, 1--106, {\it Oper. Theory Adv. Appl., 232,} Birkh\"{a}user, Basel, 2013. 

\bibitem{BdeM} L. Boutet de Movel, Boundary problems for pseudo-differential operators. {\it Acta Math.,} 126:11--51, 1971.


\bibitem{BehrndtLanger2007} J. Behrndt, M. Langer, Boundary
value problems for elliptic partial differential operators on bounded domains. {\it
J. Func. Anal.,} 243(2):536--565, 2007.

\bibitem{Berezin_Faddeev}
 F. A. Berezin, L. D.  Faddeev, Remark on the Schr\"{o}dinger
 equation with singular potential. (Russian) {\it Dokl. Akad. Nauk SSSR,} 137:1011--1014, 1961.

\bibitem{Binding_Hryniv}
P. Binding, R. Hryniv, Relative boundedness and relative compactness for linear operators in Banach spaces. {\it Proc. Amer. Math. Society,} 128(8):2287--2290, 2000.

\bibitem{MR0080271}
M.~\v{S}. Birman,
\newblock On the theory of self-adjoint extensions of positive definite
  operators.
\newblock {\em Mat. Sb. N.S.,} 38(80):431--450, 1956.

\bibitem{MR1192782}
M.~\v{S}.~Birman, M.~Z.~Solomjak,
\newblock {\em Spectral theory of selfadjoint operators in {H}ilbert space}.
\newblock Mathematics and its Applications (Soviet Series). D. Reidel
  Publishing Co., Dordrecht, 1987.
  
\bibitem{Birman}
M. S. Birman, Perturbations of the continuous spectrum of a singular elliptic operator by
varying the boundary and the boundary conditions, Vestnik Leningrad. Univ. 17 (1962), 22--55. English translation in: Spectral theory of differential operators, Amer. Math. Soc. Transl. Ser. 2, 225, Amer. Math. Soc., Providence, RI, 2008, pp. 19--53.

\bibitem{Birman_Solomyak}
M. S. Birman, M. Z. Solomiak, Asymptotics of the spectrum of variational problems on
solutions of elliptic equations in unbounded domains. {\it Funkts. Analiz Prilozhen.} 
14:27--35, 1980. English translation in: {\it Funct. Anal. Appl.} 14:267--274, 1981.
  

\bibitem{BMNW2008} M. Brown, M. Marletta, S. Naboko, I. Wood,
  Boundary triples and {$M$}-functions for non-selfadjoint
  operators, with applications to elliptic {PDE}s and block operator
  matrices. {\it J. Lond. Math. Soc. (2),} 77(3):700--718, 2008.

\bibitem{BMNW2018} M. Brown, M. Marletta, S. Naboko, I. Wood,
  The functional model for maximal dissipative operators: An
  approach in the spirit of operator knots. {\it Trans. Amer. Math. Soc.,} 373:4145-4187, 2020.

\bibitem{Pavlov_Helmholtz_resonator}
 J. Br\"{u}ning, G. Martin, B. Pavlov, Calculation of the Kirchhoff coefficients for the Helmholtz resonator.
\emph{Russ. J. Math. Phys.,} 16(2):188--207, 2009.

\bibitem{KCher}
K. D. Cherednichenko, A. V. Kiselev, Norm-resolvent convergence of
one-dimensional high-contrast periodic problems to a Kronig-Penney
dipole-type model. {\it Comm. Math. Phys.,} 349(2):441--480, 2017.

\bibitem{KCherYulia}
K. D. Cherednichenko, Yu. Yu. Ershova, A. V. Kiselev, Time-dispersive behaviour as a feature of critical contrast media,
{\it SIAM Journal on Applied Mathematics} 79(2):690--715, 2019.

\bibitem{KCherYuliaNab}
K.~Cherednichenko, Y. Ershova, A.~Kiselev, S.~Naboko,
Unified approach to critical-contrast homogenisation with explicit
links to time-dispersive media, {\it Trans. Moscow Math. Soc.} 80(2):251--294, 2019.

\bibitem{ChKS_OTAA}
\newblock K.~Cherednichenko, A.~Kiselev, L.~Silva,
\newblock Scattering theory for non-selfadjoint extensions of symmetric operators. Analysis as a Tool in Mathematical Physics: in Memory of Boris Pavlov, 194--230, {\it Oper. Theory Adv. Appl., 276,} Birkh\"{a}user, Basel, 2020. 

\bibitem{CherednichenkoKiselevSilva}
K.~D. Cherednichenko, A.~V. Kiselev, L.~O. Silva.
\newblock Functional model for extensions of symmetric operators and
  applications to scattering theory.
\newblock {\em Netw. Heterog. Media,} 13(2):191--215, 2018.

\bibitem{CherErKis}
K. D. Cherednichenko, Yu. Ershova, A. V. Kiselev,
\newblock Effective behaviour of critical-contrast PDEs: micro-resonances, frequency conversion, and time dispersive properties. I.
{\it Commun. Math. Phys.,} 375:1833--1884, 2020.

\bibitem{Anderson}
J. M. Combes, L. Thomas,
Asymptotic behaviour of eigenfunctions for multiparticle Schr\"{o}dinger operators.
{\it Comm. Math. Phys.,} 34:251--270, 1973.


  \bibitem{MR1087947}
V.~A. Derkach, M.~M.~Malamud,
\newblock Generalized resolvents and the boundary value problems for
  {H}ermitian operators with gaps.
\newblock {\em J. Funct. Anal.}, 95(1):1--95, 1991.







\bibitem{Gesztesy_Mitrea}
F. Gesztesy, M. Mitrea, A description of all self-adjoint extensions of the Laplacian and Kre\u\i n-type resolvent formulas on non-smooth domains. {\it J. Anal. Math.} 113:53--172, 2011.

\bibitem{Gor}
V. I. Gorbachuk, M. L. Gorbachuk, {\it Boundary value problems for operator differential equations.}
  Mathematics and its Applications (Soviet Series), 48,
  Kluwer Academic Publishers, Dordrecht, 1991.
  
\bibitem{Grubb_classic} G. Grubb, Singular Green operators and their spectral asymptotics. {\it Duke Math. J.,} 51(3):477--528, 1984.

\bibitem{Grubb_Robin}
G. Grubb, Spectral asymptotics for Robin problems with a discontinuous coefficient. {\it J. Spectr. Theory} 1(2):155--177, 2011.

\bibitem{Grubb_mixed}
G. Grubb, The mixed boundary value problem, Krein resolvent formulas and spectral asymptotic estimates. {\it J. Math. Anal. Appl.} 382(1):339--363, 2011.

\bibitem{KisNab}
A.~V.~Kiselev, S.~N.~Naboko, Non-self-adjoint operators with almost hermitian spectrum: matrix model. I,
{\it J. Comp. App. Math.,} 194:115--130, 2006.

\bibitem{MR0365218}
A.~N.~Ko{\v{c}}ube{\u\i},
\newblock Extensions of symmetric operators and of symmetric binary relations.
\newblock {\em Mat. Zametki,} 17:41--48, 1975.

\bibitem{MR0592863}
A.~N.~Ko\v cube\u\i,
\newblock Characteristic functions of symmetric operators
and their extensions (in Russian).
\newblock {\em Izv. Akad. Nauk Arm. SSR Ser. Mat.,} 15(3):219--232,  1980.

\bibitem{MR0407617}
T.~Kato,
\newblock {\em Perturbation theory for linear operators}.
\newblock Springer-Verlag, Berlin, second edition, 1976.
\newblock Grundlehren der Mathematischen Wissenschaften, Band 132.


\bibitem{MR0024574}
M.~Kre\u\i n,
\newblock The theory of self-adjoint extensions of semi-bounded {H}ermitian transformations and its applications. {I}.
\newblock {\em Rec. Math. [Mat. Sbornik] N.S.,} 20(62):431--495,  1947.

\bibitem{MR0024575}
M.~G. Kre\u\i n,
\newblock The theory of self-adjoint extensions of semi-bounded {H}ermitian
  transformations and its applications. {II}.
\newblock {\em Mat. Sbornik N.S.,} 21(63):365--404, 1947.

\bibitem{MR0048704}
M.~G.~Kre\u\i n,
\newblock The fundamental propositions of the theory of representations of Hermitian operators
with deficiency index {$(m,m)$}.
\newblock {\em Ukrain. Mat. \v Zurnal}, 1(2):3--66, 1949.


\bibitem{MR0217440}
P.~D.~Lax, R.~S.~Phillips,
\newblock {\em Scattering theory}.
\newblock Pure and Applied Mathematics, Vol. 26. Academic Press, New
  York, 1967.

\bibitem{livshitz} M. S. Livshitz, On a certain class of linear
  operators in Hilbert space. (Russian) \emph{Rec. Math.
  [Mat. Sbornik] N.S.,} 19(61):239--262,  1946.

\bibitem{Marletta} M. Marletta, G. Rozenblum, A Laplace operator with boundary conditions singular at one point. {\it J. Phys. A: Math. Theor.} 42, 125204, 2009.


\bibitem{MR0500225}
S.~N. Naboko,
\newblock Absolutely continuous spectrum of a nondissipative operator, and a
  functional model. {I}.
\newblock {\em Zap. Nau\v cn. Sem. Leningrad. Otdel Mat. Inst. Steklov.
  (LOMI),} 65:90--102, 204--205, 1976.
\newblock Investigations on linear operators and the theory of functions, VII.

\bibitem{MR573902}
S.~N. Naboko,
\newblock Functional model of perturbation theory and its applications to scattering theory.
\newblock In: Boundary Value Problems of Mathematical Physics 10.
\newblock {\em Trudy Mat. Inst. Steklov,} 147:86--114, 203, 1980.



\bibitem{Naimark1940}
M. Neumark, Spectral functions of a symmetric operator. (Russian)
{\it Bull. Acad. Sci. URSS. Ser. Math. [Izvestia Akad. Nauk SSSR],} 4:277--318, 1940.

\bibitem{Naimark1943}
M. Neumark, Positive definite operator functions on a commutative
group. (Russian) {\it Bull. Acad. Sci. URSS Ser. Math. [Izvestia
  Akad. Nauk SSSR],} 7:237--244, 1943.

\bibitem{Nikolski} N. K. Nikolski, {\it Operators, Functions, and Systems: An Easy Reading.} Vol. 1, 2., Mathematical Surveys and Monographs, AMS, 2002.

\bibitem{BetheSommerfeld}
L. Parnovski, A.  Sobolev, Bethe-Sommerfeld conjecture for periodic operators with strong perturbations. \emph{Invent. Math.,} 181(3):467--540, 2010.

\bibitem{MR0365199}
B.~S.~Pavlov,
\newblock Conditions for separation of the spectral components of a dissipative operator.
{\em Math. USSR Izvestija,} 9:113--137, 1975.

\bibitem{MR0510053}
B.~S. Pavlov,
\newblock Selfadjoint dilation of a dissipative {S}chr\"odinger operator, and
  expansion in its eigenfunction.
\newblock {\em Mat. Sb. (N.S.),} 102(144)(4):511--536, 631, 1977.

\bibitem{Drogobych}
B.~S.~Pavlov,
\newblock Dilation theory and the spectral analysis of non-selfadjoint differential operators.
\newblock {\em Proc. 7th Winter School, Drogobych, 1974}, TsEMI, Moscow, 2--69, 1976. English translation: {\it Transl., II Ser., Am. Math. Soc.,} 115:103--142, 1981.

 \bibitem{Pavlov_internal_structure}
B.~S.~Pavlov, A model of zero-radius potential with internal
structure. (Russian) {\it Teoret. Mat. Fiz.} 59(3):345--353,  1984.

\bibitem{RR}
M. Renardy, R. C. Rogers, {\it An Introduction to Partial Differential Equations.} Texts in Applied Mathematics 13, Springer, 2004.

\bibitem{MR822228}
M.~Rosenblum, J.~Rovnyak,
\newblock {\em Hardy Classes and Operator Theory.}
Oxford University Press,
  1985.

\bibitem{MR2330831}
V.~Ryzhov,
\newblock Functional model of a class of non-selfadjoint extensions of symmetric operators.
\newblock Operator Theory, Analysis and Mathematical Physics, 117--158, {\em Oper. Theory Adv. Appl., 174,} Birkh\"auser, Basel, 2007.

\bibitem{Ryzhov_closed}
 V.~Ryzhov, Functional model of a closed non-selfadjoint operator. {\it Integral Equations Operator Theory} 60(4):539--571, 2008.

\bibitem{Ryzh_spec} V. Ryzhov, Spectral boundary value problems and
  their linear operators, Analysis as a Tool in Mathematical Physics: in Memory of Boris Pavlov, 576--626, {\it Oper. Theory Adv. Appl., 276,} Birkh\"{a}user, Basel, 2020. 


\bibitem{MR2760647}
B.~Sz.-Nagy, C.~Foias, H.~Bercovici, L.~K{\'e}rchy,
\newblock {\em Harmonic Analysis of Operators on {H}ilbert Space}.
Springer, New York,
2010.




\bibitem{Strauss_survey} A. V. \v{S}traus, Functional models and
  generalized spectral functions of symmetric operators. {\it St. Petersburg Math. J.,} 10(5):733-784, 1999.

\bibitem{Taylor_tools} M. E. Taylor, {\it Tools for PDE:
    Pseudodifferential Operators, Paradifferential Operators, and
    Layer Potentials.} Mathematical Surveys and Monographs 81,
  American Mathematical Society, Providence, Rhode Island, 2000.



\bibitem{MR0051404} M.~I. Vi\v{s}ik, \newblock On general boundary
  problems for elliptic differential equations.  \newblock {\em Trudy
    Moskov. Mat. Ob\v{s}\v{c}.,} 1:187--246, 1952.



\end{thebibliography}
\end{document}